\documentclass[12pt]{article}

\usepackage{latexsym,amsmath,amsfonts}

\newtheorem{theorem}{Theorem}
\newtheorem{proposition}{Proposition}
\newtheorem{corollary}{Corollary}

\newtheorem{lemma}{Lemma}
\newtheorem{question}{Question}

\newtheorem{assumption}{Assumption}
\newtheorem{definition}{Definition}

\newtheorem{claim}{Claim}

\newtheorem{remark}{Remark}

\newcommand\ignore[1]{}

\def\R{\mathbb{R}} 
\def\N{\mathbb{N}} 

\newcommand{\Gnote}[1]{}

\newcommand{\Ex}[1]{\mathbb{E}\left[#1\right]} 

\newcommand{\Ind}[1]{\mathbb{I}_{#1}} 

\renewcommand{\Pr}[1]{\mathbb{P}\left(#1\right)} 

\def\sM{\mathcal{M}}

\def\sX{\mathcal{X}}

\newcommand\QED{\ifhmode\allowbreak\else\nobreak\fi
\quad\nobreak$\Box$\medbreak}
\newcommand{\proofstart}{\par\noindent\sl Proof:\rm\enspace}
\newcommand{\proofend}{\QED\par}
\newenvironment{proof}{\proofstart}{\proofend}

\def\eps{\epsilon}

\def\ER{Erd\H{o}s-R\'{e}nyi}
\newcommand{\modulo}[1]{\left| #1\right|}
\newcommand{\pnorm}[2]{\left\|#2\right\|_{#1}}
\newcommand{\supnorm}[1]{\pnorm{\infty}{#1}}

\newcommand{\expp}[1]{\exp \left( #1 \right)} 
\newcommand{\prodint}[2]{ \left\langle #1,#2 \right\rangle }
\newcommand{\one}{{\bf 1}}

\begin{document}

\title{Interacting diffusions on random graphs with diverging average degrees: hydrodynamics and large deviations}


\author{Roberto I. Oliveira\thanks{IMPA, Rio de Janeiro, Brazil. 22460-320. \texttt{rob.oliv@gmail.com},~\texttt{rimfo@impa.br}.  supported by a Bolsa de Produtividade em Pesquisa from CNPq, Brazil. His work in this article is part of the activities of FAPESP Center for Neuromathematics (grant \# 2013/07699-0, FAPESP - S. Paulo Research Foundation). }        \and
        Guilherme H. Reis\thanks{IMPA, Rio de Janeiro, Brazil. 22460-320. \texttt{ghreis@impa.br}. Supported by a Ph.D. scholarship from CNPq, Brazil (grant \# 140768/2015-7.)}  
}

\maketitle

\begin{abstract} 
We consider systems of mean-field  interacting diffusions, where the pairwise interaction structure is described by a sparse (and potentially inhomogeneous) random graph. Examples include the stochastic Kuramoto model with pairwise interactions given by an Erd\H{o}s-R\'{e}nyi graph. Our problem is to compare the bulk behavior of such systems with that of corresponding systems with dense nonrandom interactions. For a broad class of interaction functions, we find the optimal sparsity condition that implies that the two systems have the same hydrodynamic limit, which is given by a McKean-Vlasov diffusion. Moreover, we also prove matching behavior of the two systems at the level of large deviations. Our results extend classical results of dai Pra and den Hollander and provide the first examples of LDPs for systems with sparse random interactions. 
\end{abstract}

\section{Introduction}
\label{intro}

Mean-field models of interacting diffusion processes have attracted much interest. Physically, they are models for systems with many interacting components that can range from the brain to electrical circuits \cite{surveyKuramoto,kuramoto2012chemical}. Mathematically, they give rise to interesting phenomena, such as equations of McKean-Vlasov type \cite{Sznitman_Chaos,Pra1996}. 

Classical models typically have pairwise interactions between all or most pairs of diffusions. In this paper we consider certain systems with sparse disordered interactions. For a simple concrete example, take a large $n\in\N$ and consider a random symmetric matrix
\[A^{(n)} = (A_{i,j}^{(n)})_{i,j=1}^n\in \{0,1\}^{n\times n}\]
whose entries for $1\leq i\leq j\leq n$ are i.i.d. Bernoulli random variables with mean $p(n)$. This matrix can be thought of as the adjacency matrix of an \ER~random graph $G(n,p(n))$ \cite{bollobas01}, except that we allow ``loops" (self-edges). The reader should think that $p(n)\to 0$ as $n\to +\infty$.

Now consider two systems of interacting diffusions,
\[\overline{\theta}^{(n)}:=(\overline{\theta}_i^{(n)})_{i\in[n]}\mbox{ and }{\theta}^{(n)}:=({\theta}_i^{(n)})_{i\in[n]},\]
defined for times $0\leq t\leq T$ via the  stochastic differential equations:
\begin{eqnarray}\label{eq:kuramotointro}d\overline{\theta}^{(n)}_i(t) &=& \left(\sum_{j=1}^n\,\frac{\kappa}{n}\sin(\overline{\theta}_j^{(n)}(t)-\overline{\theta}_i^{(n)}(t)) + \omega_i^{(n)}\right)dt + dB^{(n)}_i(t);\\
\label{eq:kuramotointrosparse}d{\theta}^{(n)}_i(t) &=& \left(\sum_{j=1}^n\,\frac{\kappa A^{(n)}_{i,j}}{np(n)}\sin({\theta}_j^{(n)}(t)-{\theta}_i^{(n)}(t)) + \omega_i^{(n)}\right)dt + dB^{(n)}_i(t).\end{eqnarray}
Here, the $\omega_i^{(n)}$ are random ``natural frequencies" and the $B^{(n)}_i$ are independent standard Brownian motions. 

The model in (\ref{eq:kuramotointro}) is the stochastic version of the standard Kuramoto model, a family of widely studied models of synchronization \cite{Pra1996,surveyKuramoto,generalisingKuramoto,kuramoto2012chemical}. This model has a {\em dense interaction structure}, in that the drift term for each diffusion contains terms involving all other diffusions. By contrast, if $p(n)\ll 1$, then pairwise interactions in (\ref{eq:kuramotointrosparse}) are {\em sparse and random}. This may provide a more realistic model for many systems of interest, where connections are disordered and not abundant.

One may ask what relationship (if any) there exists between the properties of (\ref{eq:kuramotointro}) and (\ref{eq:kuramotointrosparse}). Some recent papers have shown that, if the random graph is not too sparse, then the two systems have similar bulk behavior in the thermodynamic limit. For instance, Delattre, Giacommin and Lu\c{c}on \cite{Delattre2016} prove such a result when $np(n)\gg \log n$, and Medvedev \cite{medvedev-sparse} does so for $np(n)\gg \sqrt{n}$. This raises two natural questions.

\begin{question}\label{question:optimal}What is the optimal sparsity condition on the random graph that leads systems (\ref{eq:kuramotointro}) and (\ref{eq:kuramotointrosparse}) to have the same hydrodynamic limits?\end{question}

\begin{question}\label{question:other}Can this similarity be extended to other aspects of bulk behavior, like fluctuations and large deviations?\end{question}

Clearly, the same questions can be asked about many other systems beyond the Kuramoto case.

\subsection{Our contribution}

The results in this paper gives a fairly complete answer to Question \ref{question:optimal} and obtains large deviations results in the direction of Question \ref{question:other}, for a broad class of systems. 

For concreteness, we first state our result in the Kuramoto case. Assume the two systems of diffusions have the same initial conditions and are built from the same Brownian motions. Define the double-layer empirical measures \cite{Pra1996}: \[\overline{L}_n:=\frac{1}{n}\sum_{i=1}^n\delta_{(\overline{\theta}_i^{(n)},\omega_i^{(n)})} \mbox{ and }{L}_n:= \frac{1}{n}\sum_{i=1}^n\delta_{(\theta_i^{(n)},\omega_i^{(n)})},\]
where $\theta_i^{(n)},\overline{\theta}_i^{(n)}\in C([0,T],\R)$ are the trajectories of individual particles. Our main finding -- contained in Theorem \ref{thm:main} in Section \ref{sec:main} below -- is that:
\begin{center}When $np(n)\to +\infty$, $L_n$ and $\overline{L}_n$ obey the same large deviations principle.\end{center}

This implies in particular that the system with sparse random interactions has the same McKean-Vlasov limit as the dense system. Moreover, we obtain what are (to the best of our knowledge) the first LDP in the sparse random setting. 

The condition $np(n)\to +\infty$ implies diverging average degree in the random graph. As it turns out this condition is optimal:  the same system with $np(n)\to c\in\R$ would have a different limit (we study this regime in a companion paper in preparation). In this sense, our Theorem fully answers Question \ref{question:optimal}. 

In fact, Theorem \ref{thm:main} gives optimal results beyond (\ref{eq:kuramotointro}) and (\ref{eq:kuramotointrosparse}). In terms of interaction functions, Theorem \ref{thm:main} covers a slightly more general setting than the  ``Hamiltonian interactions"  considered in the classical paper of dai Pra and den Hollander \cite{Pra1996}. 

As for random graph models, we will consider sparse versions of the $W$-random graphs from the theory of graph limits \cite{Lovasz_Limits} (see also \cite{medvedev-meanfield} for the dense case). To define one such graph, associate a vector $\omega^{(n)}$ of ``media variables"~to the $n$ particles. We then let the probability of an edge between particles $i$ and $j$ has the form \[\Pr{A^{(n)}_{i,j}=1\mid \omega^{(n)}} = p(n)W(\omega_i^{(n)},\omega_j^{(n)})\] for a function $W$ and a sequence $p(n)\in (0,1]$. The main attraction of this model is that it is inhomogeneous: different potential edges can be more or less likely to appear. Nevertheless, the condition that $np(n)\to +\infty$ is still necessary and sufficient for comparison with the dense setting.

\subsection{Discussion and further background}\label{sub:background}
In what follows we give a very selective survey of results on interacting diffusions and relate them to our own work. 

Models with dense mean-field interactions are classical. Sznitman's lecture notes \cite{Sznitman_Chaos} give an early overview of rigorous work in the area. For our purposes, the paper \cite{Pra1996} by dai Pra and den Hollander is especially important, as it proves fairly general results on Large Deviations (which we employ in this paper) and Central Limit Theorems (which we do not pursue). Lu\c{c}on \cite{lucon2017} obtains quenched large deviations in a similar setting. Budhiraja, Dupuis and Fischer \cite{budhiraja2012} consider a larger family of interactions that includes jumps, non-constant diffusion coefficients and nonlinear terms in the empirical measure. However, \cite{budhiraja2012} does not consider ``media variables" or ``impurities". 

Recent papers have considered sparse, disordered and/or geometrically constrained interactions. Neuroscience provides an important impetus for these studies. For instance, Lu\c{c}on and Stannat \cite{lucon2014,lucon2016} derive hydrodynamic limits and fluctuations for geometrically constrained models with singular interactions. Cabana and Touboul \cite{Cabana2013,Cabana20181,Cabana} consider models with interaction delays and Gaussian couplings that have highly nontrivial behavior in the thermodynamic limit.

The closest results to our own work are \cite{medvedev-meanfield,medvedev-sparse,chibamedvedevmizuhara2018,Delattre2016,Coppini2018}. These papers deal with hydrodynamic limits of models with random interactions. Except for the very recent \cite{Coppini2018}, they do not obtain results on large deviations. 

Chiba and Medvedev \cite{medvedev-meanfield} consider the Kuramoto model with no noise over dense $W$-random graphs. They describe the bulk behavior of such system, and study the transition points for limiting Vlasov PDE. Chiba et al. \cite{chibamedvedevmizuhara2018} contains numerical results on bifurcations of the limiting model. Medvedev \cite{medvedev-sparse} obtains results for sparser inhomogeneous and possibly directed graphs. In the Erd\"{o}s-R\'{e}nyi setting, he assumes $p(n)=n^{-\gamma}$ for $\gamma>0.5$, which is considerably stronger than $np(n)\to +\infty$. Our techniques can be adapted to the directed case (cf. Remark \ref{rem:LDPdifferences} below). As an aside, note that these papers consider noiseless systems where the Brownian motions are absent. Our methods could cover this. However, in this case the dense system is deterministic and it would not make sense to compare its LDP to the sparse case.

Delattre, Giacomin and Lu\c{c}on \cite{Delattre2016} construct a coupling between {\em individual particles} in the finite-$n$ model and a set of independent McKean-Vlasov diffusions. This leads to a hydrodynamic limit for certain sparse systems, but not to large deviations. Unlike our paper, they do not impose a distributional assumption on the interaction graph: they only require that it is nearly regular with large degree. However, in the particular case of Erd\H{o}s-R\'{e}nyi graphs, this leads to the condition $np(n)(\log n)^{-1}\to +\infty$, which is stronger than what we require.

The very recent preprint by Coppini, Dietert and Giacomin \cite{Coppini2018} appeared only a few days after the first version of the present paper. The authors obtain a hydrodynamic limit and a LDP over the Erd\H{o}s-R\'{e}nyi random graph under the condition that $\liminf np(n)(\log n)^{-1}>0$. Unlike our main result, their theorem is a {\em quenched} statement with respect to the initial conditions. In addition, they consider a more general class of interactions. It should be possible to apply a modification of our Lemma \ref{lem:expequivuse} to reprove their result under the optimal condition $np(n)\to +\infty$.

We finish this section by highlighting some aspects of our proofs. The main technical step will be to show that the measures $L_n$ and $\overline{L}_n$ are {\em exponentially equivalent} in the sense that, for a suitable metric $d$ over probability measures, $\Pr{d(L_n,\overline{L}_n)>\eta}$ goes to $0$ faster than any exponential function for any fixed $\eta>0$. The role of this concept, explained in Section \ref{sec:proofmain} below, is that exponential equivalence allows us to transfer the LDP from one system to another. We can then apply the LDP by \cite{Pra1996} with slight extensions discussed in Remark \ref{rem:LDPdifferences} and Appendix \ref{sec:extendingg} below.

To prove exponential equivalence, a crucial step is to bound the difference between the adjacency matrix of the random graph and its entrywise expectation. Bounding the spectral norm of this difference would be natural, but this norm does not behave well when $np(n)\ll \log n$ due to large degree vertices (see e.g. \cite[Remark 4.2]{Guedon2016}). It turns out that the following weaker norm is sufficient for our argument to go through:
\begin{eqnarray*}\|A^{(n)} - \Ex{A^{(n)}}\|_{\infty\to 1} &:=& \sup\{\|(A^{(n)} - \Ex{A^{(n)}})\vec{x}\|_1\,:\,\vec{x}\in \R^n,\,\|\vec{x}\|_{\infty}\leq 1\} \\ 
&=& \sup\{\prodint{\vec{y}}{(A^{(n)} - \Ex{A^{(n)}})\vec{x}}\,:\,\vec{x},\vec{y}\in [-1,1]^n\}.\end{eqnarray*}
Unlike the spectral norm, this norm is ``small" whenever $np(n)\to +\infty$. This was observed by Gu\'{e}don and Vershynin in the context of community detection in random graphs \cite[Remark 4.2]{Guedon2016}. Noticing that this is the right norm for our problem is one of our main contributions.

\subsection{Organization}

The remainder of the paper is organized as follows. Section \ref{sec:prelim} fixes notation and recalls known results. Section \ref{sec:models} details our assumptions and establishes the framework for the remainder of the paper. To illustrate the assumptions, we also give an example that is a somewhat more sophisticated than the Kuramoto model in the Introduction. Section \ref{sec:main} contains a description of McKean-Vlasov diffusions and the statement and proof of our main Theorem. This proof relies on the exponential equivalence result that is stated and proved in Section \ref{sec:expequiv}. The main lemmas in that proof are also proven in that section. Some auxiliary Lemmas are left to Section \ref{sec:auxiliary}. The appendices contains a technical approximation lemma and an argument for extending the LDP of dai Pra and den Hollander \cite{Pra1996} to a slightly larger class of interaction functions.

\section{Preliminaries}\label{sec:prelim}

In this paper $\N$ is the set of positive integers. For $n\in\N$, $[n]:=\{1,2,\dots,n\}$. 

Let $(S,\mathcal{S})$ be a measurable space and $P$ be a probability measure over $(S,\mathcal{S})$. We write $X\sim P$ to mean that $X$ is a random element of $(S,\mathcal{S})$ with law $P$. The product of probability measures $P$ and $Q$ is denoted by $P\otimes Q$. We also write: \[P^{\otimes n} = \underbrace{P\otimes P\otimes \dots \otimes P}_{n\text{ times}}.\]

Given a metric space $(S,d)$ and a function $f:S\to \R$, we define:
\begin{eqnarray}\|f\|_{\infty}&:=& \sup\{|f(x)|\,:\,x\in S\};\\
\|f\|_{Lip} &:=& \sup\left\{\frac{|f(x)-f(y)|}{d(x,y)}\,:\,x,y\in S,\,x\neq y\right\};\\
\|f\|_{BL} &:=&2( \|f\|_{\infty} + \|f\|_{Lip}).\end{eqnarray}
We say that $f$ is Lipschitz if $\|f\|_{Lip}<+\infty$ and bounded Lipschitz if $\|f\|_{BL}<+\infty$.
\begin{remark}\label{rem:bldist} In this setting,
\[\pnorm{BL}{h}\leq 1\implies |h(x)-h(y)|\leq |x-y|\wedge 1.\]
\end{remark}

Now let $(S,d)$ be a Polish metric space, with $\mathcal{S}$ the Borel $\sigma$-field. We consider the space $\mathcal{M}_1(S)$ of probability measures over $(S,\mathcal{S})$. The topology of weak convergence in that space is metrized by the BL metric, defined for $P,Q\in \mathcal{M}_1(S)$ as follows:
\[d_{BL}(P,Q):= \sup\left\{\left|\int_S\,f\,d(P-Q)\right|\,:\, f:S\to \R\mbox{ with }\|f\|_{\rm BL}\leq 1\right\}.\]
$(\mathcal{M}_1(S),d_{BL})$ is a Polish metric space. We also consider the Wasserstein metric:
\[d_{W}(P,Q):=\sup\left\{\left|\int_S\,f\,d(P-Q)\right|\,:\, f:S\to \R\mbox{ $1$-Lipschitz}\right\},\]
which is only defined for $P$ and $Q$ with finite first moments. Clearly, $d_{BL}\leq d_W$ always.

We recall the definition of a Large Deviations Principle (cf. \cite[Section~1.2]{dembo2009large} ). 

\begin{definition}[Large Deviations Principle]\label{def:LDP}A good rate function $I$ is a lower semicontinuous mapping $I:S \to [0,\infty]$ such that the level sets $I^{-1}((-\infty,a])$ are compact. A sequence $\{X_n\}_{n\in\N}$ of random elements of $S$  satisfies a Large Deviation Principle (LDP) with  rate function $I$ and speed $n$ if, for all Borel-measurable $E\subset S$
\[-\inf_{x \in {\rm int} E}I(x) \leq \liminf_{n \to \infty}\dfrac{1}{n} \log \Pr{X_n\in E}\leq \limsup_{n \to \infty}\dfrac{1}{n}\log  \mu_n(E) \leq -\inf_{x \in \overline{E}}I(x).\]
 \end{definition}
 
A slightly confusing point is that oftentimes our space $S$ will be the space $\sM_1(X)$ with the metric $d_{BL}$ for some other metric space $(X,\rho)$.

\section{The models}\label{sec:models}

In this section we fully specify the interacting diffusion models we will consider. More specifically, for each $n\in\N$ and each index $i\in[n]$ we will define:  
\[(\theta_i^{(n)},\overline{\theta}^{(n)}_i,\omega^{(n)}_i)\]
where $\theta_i^{(n)},\overline{\theta}^{(n)}_i\in C([0,T],\R)$ are coupled diffusion processes and $\omega_i^{(n)}\in\R^d$ are ``media variables"~that represent individual properties of the interacting units. 

\subsection{Definition}\label{sub:definition}

Fix $d\in\N$, a time horizon $T>0$ and a sequence $\{p(n)\}_{n\in\N}\subset (0,1]$. To define the model, we need the following ingredients.
\begin{enumerate}
\item A probability distribution $\lambda$ over $\R$ for the {\em initial states of the diffusions}.
\item A probability distribution $\mu$ over $\R^d$ for the {\em media variables}.
\item A function $\phi:\R\times \R\times \R^d\times \R^d\to \R$ that determines {\em pairwise interactions between particles}. These terms will depend on the positions of the diffusions and on their media variables.
\item A function $\psi:\R\times \R^d\to \R$ that determines {\em single-particle drift terms}. These terms depend on the position of the particle and on its media variable.
\item A function $W:\R^d\times \R^d\to [0,+\infty)$ that will determine the {\em edge probabilities} in our random graph models together with the parameters $p(n)$. We assume $p(n)\|W\|_{\infty}\leq 1$ and that $W(a,b)=W(b,a)$ for all $a,b\in\R^d$. 
\item Finally, we define:
\begin{eqnarray}\label{eq:defphibar}\overline{\phi}(x,y,\omega,\pi)&:=& W(\omega,\pi)\,\phi(x,y,\omega,\pi)\\ \nonumber & & ((x,y,\omega,\pi)\in\R\times \R\times \R^d\times \R^d).\end{eqnarray}
\end{enumerate}   

Some technical conditions on these ``ingredients" will be given in the next subsection. Postponing them, we first give the definition of the diffusions.

Let $\mathcal{W}$ denote the standard Wiener measure over $C([0,T],\R)$. To define our model for a given $n\in\N$, we first sample independent random vectors:
\[\vec{\xi}^{(n)} = (\xi_i^{(n)})_{i\in [n]}\sim \lambda^{\otimes n},\, \vec{\omega}^{(n)} = (\omega_i^{(n)})_{i\in [n]}\sim \mu^{\otimes n},\,\vec{B}^{(n)} = (B_i^{(n)})_{i\in [n]}\sim \mathcal{W}^{\otimes n}.\]
Conditionally on these choices, we define a random $n\times n$ symmetric matrix \[A^{(n)}\in \{0,1\}^{n\times n}\] as follows: the entries  $A^{(n)}_{i,j}$ with $1\leq i\leq j\leq n$ are independent, with 
\[\forall (i,j)\in [n]^2\,:\,\Pr{A^{(n)}_{i,j}=1\mid \vec{\xi}^{(n)},\vec{\omega}^{(n)},\vec{B}^{(n)}} = p(n)\,W(\omega_i^{(n)},\omega_j^{(n)}).\]
We interpret $A^{(n)}$ as the adjacency matrix of a random graph on the $n$ particles.

We now define our coupled systems of interacting diffusions as follows. 

\begin{definition}\label{def:process}Given the above (random) choices of
\[\vec{\xi}^{(n)},\,\vec{\omega}^{(n)},\,\vec{B}^{(n)}\mbox{ and }A^{(n)},\]
the two systems of interacting diffusions
\[\theta^{(n)}:=(\theta_i^{(n)})_{i\in[n]}\mbox{ and }\overline{\theta}^{(n)}:=(\overline{\theta}_i^{(n)})_{i\in[n]}\]
are defined below. 
\begin{enumerate}
\item $\theta^{(n)}$ is a strong solution of the following system of SDEs:
\[\left\{\begin{array}{lcl}d\theta_i^{(n)}(t) &=& \left(\frac{1}{np(n)}\sum_{j=1}^nA_{i,j}^{(n)}\,\phi(\theta_i^{(n)}(t),\theta_j^{(n)}(t),\omega_i^{(n)},\omega_j^{(n)})\right)\,dt \\ & &+ \psi(\theta_i^{(n)}(t),\omega_i^{(n)})\,dt + dB_i^{(n)}(t) \\ & &  (0\leq t\leq T,\, i\in[n]); \\ \theta^{(n)}(0) &=& \xi^{(n)};\end{array}\right.\]
\item $\overline{\theta}^{(n)}$ is a strong solution of the following system of SDEs:
\[\left\{\begin{array}{lcl}d\overline{\theta}_i^{(n)}(t) &=& \left(\frac{1}{n}\sum_{j=1}^n\overline{\phi}(\overline{\theta}_i^{(n)}(t),\overline{\theta}_j^{(n)}(t),\omega_i^{(n)},\omega_j^{(n)})\right)\,dt \\ & & + \psi(\overline{\theta}_i^{(n)}(t),\omega_i^{(n)})\,dt + dB_i^{(n)}(t) \\ & &  (0\leq t\leq T,\, i\in[n]); \\ \overline{\theta}^{(n)}(0) &=& \xi^{(n)}.\end{array}\right.\]
\end{enumerate}
We also define the double-layer empirical measures of the two systems:
 \[{L}_n:= \frac{1}{n}\sum_{i=1}^n\delta_{(\theta_i^{(n)},\omega_i^{(n)})}\mbox{ and }\overline{L}_n:=\frac{1}{n}\sum_{i=1}^n\delta_{(\overline{\theta}_i^{(n)},\omega_i^{(n)})},\]
 which are random elements of the space $\mathcal{M}_1(C([0,T],\R)\times \R^d)$.\end{definition}

The above systems of diffusions have unique strong solutions whenever the functions $\phi,\psi$ are bounded and Lipschitz (we make stronger assumptions below). In our definition the two systems of diffusions are naturally {\em coupled}: they have identical initial conditions and are defined with respect to the same Brownian motions. 

\subsection{Technical assumptions}\label{sec:assumptions}

We now clarify the technical assumptions we will need for our arguments. For later reference, we repeat some of the statements already made above. 

Our first assumption is about the probability measures $\mu$ and $\lambda$.
\begin{assumption}[Starting and media measures]\label{ass:medvartail} We assume $\mu$ (the distribution of the media variables) is a probability measure over $\R^d$. The measure $\lambda$ (for the initial conditions) is a probability measure over $\R$ with a density: \[\rho_{\lambda}\in L^1(\R,dx)\cap L^p(\R,dx)\]
for some $p>1.$\end{assumption}

The second assumption constrains the function $W$ that determines the edge probabilities.
\begin{assumption}\label{ass:edge}The function $W:\R^d\times \R^d\to [0,+\infty)$ is bounded, symmetric, Lipschitz and does not change with $n$. Moreover, $p(n)\,\|W\|_{\infty}\leq 1$ (so that $p(n)\,W(\omega,\pi)\in[0,1]$ always). \end{assumption}

Finally, in order to apply the methods and results of \cite{Pra1996}, we need our interactions to satisfy a version of their Hamiltonian condition. We comment on this condition below. 
\begin{assumption}[Hamiltonian interactions]\label{ass:fg}Given functions $\phi_0:\R\times \R^d\times \R^d\to\R$ and $\psi:\R\times \R^d\to \R$, we use primes to denote derivatives in the first variable. We make the following assumptions:
\begin{enumerate}
 \item $\phi_0$ and $\phi'_0$ are both bounded and Lipschitz continuous in all variables. Moreover, $\phi_0$ is an odd function in the first variable, in that \[\phi_0(x,\omega,\pi)=-\phi_0(-x,\omega,\pi)\] for all $(x,\omega,\pi)\in\R\times \R^d\times \R^d$. We also assume $\phi_0(x,\omega,\pi)$ is symmetric in $\omega$ and $\pi$. We let
 \[\phi(x,y,\omega,\pi):=\phi_0(y-x,\omega,\pi)\,\,((x,y,\omega,\pi)\in\R\times\R\times \R^d\times \R^d).\]
Let $f$ be an indefinite integral of $-\phi_0$ in the first variable (so that $f'=-\phi_0$) and define:
\begin{equation}\label{eq:deffbar}\overline{f}(x,\omega,\pi):= W(\omega,\pi)\,f(x,\omega,\pi)\,\,((x,\omega,\pi)\in\R\times \R^d\times \R^d).\end{equation}
We assume that $\overline{f}$ is Lipschitz, and note that $\overline{f}',\overline{f}''$ are bounded Lipschitz (because $f',f''$ and $W$ are bounded Lipschitz).
 \item $\psi$ and $\psi'$ are both bounded and Lipschitz continuous in all variables. We let $g$ denote an indefinite integral of $-\psi$ in the first variable (so that and note that $g$ is Lipschitz with $g',g''$ bounded and Lipschitz.
 \end{enumerate}
Finally, we define the Hamiltonian:
 \[\overline{H}_n(x^{(n)},\omega^{(n)}):= \frac{1}{2n}\sum_{i,j=1}^n\,\overline{f}(x^{(n)}_i-x^{(n)}_j,\omega_i^{(n)},\omega_j^{(n)}) + \sum_{i=1}^n\,g(x^{(n)}_i,\omega_i^{(n)}).\]
 \end{assumption}
 
 \begin{remark}\label{rem:LDPdifferences}The main reason for this definition is that the evolution $\overline{\theta}_i^{(n)}(t)$  takes the form of a gradient evolution with noise (compare with Definition \ref{def:process}):
\[d\overline{\theta}_i^{(n)}(t) = -\partial_{x_i^{(n)}}\overline{H}_n(\overline{\theta}^{(n)}(t),\omega^{(n)}) + dB_i^{(n)}(t).\]

For this we do not quite need that $W$ is symmetric, but only that the function $\phi(x,y,\omega,\pi)$ takes the form:
\[\phi(x,y,\omega,\pi) = W(\omega,\pi)\,f'(x-y,\omega,\pi) - W(\pi,\omega)\,f'(y-x,\pi,\omega)\]
for some $f$. The symmetry of $W$ is natural when interactions are described by an unoriented graph, but some papers consider oriented interactions as well \cite{medvedev-sparse}. It is not hard to modify our proof to cover this.

Under our Assumptions \ref{ass:medvartail}, \ref{ass:edge}, and \ref{ass:fg}, \cite{Pra1996} derive a McKean-Vlasov limit and a LDP for what we call $\overline{L}_n$ under the assumption that $\overline{f},\overline{f}',\overline{f}''$ and $g,g',g''$ are all bounded Lipschitz. By contrast, we only assume that $\overline{f},g$ are Lipschitz and $\overline{f}'\overline{f}'',g',g''$ are bounded Lipschitz. In Appendix \ref{sec:extendingg} we show how small modifications of the proofs of \cite{Pra1996} imply that that our weaker assumptions imply their result. \end{remark}

\begin{remark}\label{rem:whyassumehamiltonian}We assume Hamiltonian interactions because this is a case where LDPs have been proven for $\overline{L}_n$ -- the empirical measure over trajectories -- in the dense setting. Our proof methods imply that, whenever $\psi,\phi$ are bounded and Lipschitz continuous, then $L_n$ and $\overline{L}_n$ are exponentially equivalent even if interactions are not Hamiltonian. Therefore, any LDP result for $\overline{L}_n$ under more general conditions on the interactions would translate into a more general LDP for $L_n$. This is a consequence of the concept of exponential equivalence described in the proofs of the main result (Theorem \ref{thm:main}) and Theorem \ref{theo:expequigenphipsi}.

There are settings, like that of Budhiraja et al. \cite{budhiraja2012} where an LDP is only known for the ``flow empirical measure"  of pairs $(\overline{\theta}_i^{(n)}(t),\omega_i^{(n)})$ for each $t\geq 0$. The flow measure contains less information than $\overline{L}_n$, but is also an interesting object of study. Our techniques can probably be used to derive a LDP for the flow empirical measure over sparse graphs in the setting of \cite{budhiraja2012}, at least when the interaction functions and diffusion coefficients are bounded Lipschitz. in our setting, we also need that the drift term be linear in the empirical measure.\end{remark}

\subsection{An example: a spatially extended Kuramoto model}\label{sec:example}

Our framework encompasses many examples. For concreteness, we present in detail a spatially extended version of the Kuramoto model with sparse random interaction structure. For simplicity, we consider the model only in $3$ spatial dimensions. 

We let $d=4$ be the dimension of the media variables and write each $\omega\in \R^{4}$ as $(\omega_s,\omega_f)$ with $\omega_s=(\omega_x,\omega_y,\omega_z)\in\R^3$. We interpret $\omega_s$ as the spatial location of a particle and the last coordinate $\omega_f$ as a ``natural frequency". For simplicity, we assume $\mu$ is the uniform measure over $[0,1]^4\subset \R^4$ and that the measure $\lambda$ for the initial conditions has a density with bounded support. 

We assume our random connections in our interaction graph have a probability that decays with distance.
\[W(\omega,\pi):= \frac{1}{1+C|\omega_s-\pi_s|^{\alpha}}\,\,(\omega,\pi\in\R^4)\]
where $C,\alpha\geq 0$ are constants.
For $(x,y,\omega,\pi)\in\R\times \R\times \R^{d}\times \R^d$, we define:
\begin{eqnarray}\phi(x,y,\omega,\pi)&:=&\kappa\,\sin(y-x);\\
\psi(x,\omega) &:=& \omega_f,\end{eqnarray} 
where $\kappa$ denotes the coupling strength.

We choose some sequence $p(n)\to 0$ with $np(n)\to +\infty$ (e.g. $p(n)=\log\log(n+10)/n$ for large enough $n$). A connection between particle $i$ and $j$ exists with probability
\[\Pr{A^{(n)}_{i,j}=1\mid \vec{\xi}^{(n)},\vec{\omega}^{(n)},\vec{B}^{(n)}}= \frac{p(n)}{1+C|\omega_{i,s}^{(n)}-\omega^{(n)}_{j,s}|^{\alpha}}.\]
The evolution equations for our systems are:
\[\begin{array}{lcl}d\theta_i^{(n)}(t) &=& \left(\frac{\kappa}{np(n)}\sum_{j=1}^nA_{i,j}^{(n)}\,\sin(\theta_j^{(n)}(t)-\theta_i^{(n)}(t))\right)\,dt\\ & & + \omega_i^{(n)}\,dt + dB_i^{(n)}(t)\\ d\overline{\theta}_i^{(n)}(t) &=& \left(\frac{\kappa}{n}\sum_{j=1}^n\frac{1}{1+C|\omega_{i,s}^{(n)}-\omega^{(n)}_{j,s}|^{\alpha}}\,\sin(\overline{\theta}_j^{(n)}(t)-\overline{\theta}_i^{(n)}(t))\right)\,dt\\ & & + \omega_i^{(n)}\,dt + dB_i^{(n)}(t).\end{array}\]

It is easy to check that all of our assumptions are satisfied by this example. Its structure is typical of examples of our main result, in that it has the following properties.
\begin{enumerate}
\item Each particle has $\approx np(n)$ neighbors on average (e.g. $\approx \log\log n$). This means the interaction graph can be quite sparse. 
\item Edge probabilities decay with the distance between particles, but remain bounded away from $0$. This captures spatially extended features, but does not allow for e.g. only short range interactions (contrast this with \cite{lucon2014,lucon2016}). 
\end{enumerate}

\section{Main result: McKean-Vlasov limit and LDP}\label{sec:main}

This section presents our main theorem. We start with the definitions of a McKean-Vlasov diffusion. We then state and prove our main Theorem (modulo many later results). 

\subsection{McKean-Vlasov diffusions}

We start with some notation. Given $\nu\in \mathcal{M}_1(C([0,T],\R)\times \R^d)$, we write $\nu_{media}(d\omega)$ for the second marginal of this measure, and $\nu_{process}(d\theta\mid\omega)$ for the conditional law of the first coordinate. With this notation, we have the disintegration formula:
\[\nu(d\theta,d\omega) = \nu_{media}(d\omega)\,\nu_{process}(d\theta\mid\omega).\]
Also, if $(\theta,\omega)\sim \nu$ and $0\leq t\leq T$, we let $\Pi_t\nu$ denote the law of the pair $(\theta(t),\omega)\in \R\times \R^d$. 

Fix interaction functions $\psi,\phi,W$ as in Assumption \ref{ass:fg} and \ref{ass:edge}, and measures $\lambda$ (for initial conditions) and $\mu$ (for media variables) as in Assumption \ref{ass:medvartail}. We define a mapping
\[\nu\in \mathcal{M}_1(C([0,T],\R)\times \R^d)\mapsto P^{\nu}\in\mathcal{M}_1(C([0,T],\R)\times \R^d).\] Given $\nu$, $P^{\nu}$ is defined as follows:
\begin{enumerate}
\item $P^{\nu}_{media}=\mu$ is the measure  we have chosen for the media variables;
\item For $\mu$-a.e. $\omega\in\R^d$, $P^\nu_{process}(d\theta\mid \omega)$ is the law of a Markov diffusion process $\Theta^{\omega}$ with $\Theta^\omega(0)\sim \lambda$ and 
\[d\Theta^{\omega}(t):= \left(\int_{\R\times \R^d}\overline{\phi}(\Theta^{\omega}(t),y,\omega,\pi)\,\Pi_t\nu(dy,d\pi) + \psi(\Theta^{\omega}(t),\omega)\right)dt + dB(t)\]
for $0\leq t\leq T$. Here $B(\cdot)$ is a standard Brownian motion.
\end{enumerate}

For the next definition, we recall that the relative entropy of two measures $P$, $Q$ over the same measurable space $(X,\sX)$ is:

\[H(P\mid Q) = \left\{\begin{array}{ll}\int_{X}\ln\left(\frac{dP}{dQ}\right)\,dP, & \mbox{ if }P\ll Q; \\ +\infty, & \mbox{otherwise.}\end{array}\right.\]

\begin{definition}[McKean-Vlasov Diffusion]\label{def:rate} We say $Q_*\in \mathcal{M}_1(C([0,T],\R)\times \R^d)$ is a {\em McKean-Vlasov diffusion} (for this choice of $\psi,\phi,W,\lambda$ and $\mu)$ if $Q_*=P^{Qˆ_*}$. We also set:
\[I(\nu):=H(\nu\mid P^{\nu})\,\,(\nu\in \mathcal{M}_1(C([0,T],\R)\times \R^d))\]
and note that the McKean-Vlasov diffusions are precisely the zeros of this function.\end{definition}

Sufficient conditions for existence and uniqueness of a McKean-Vlasov diffusion $Q_*$ are given in \cite{lucon2011,Sznitman_Chaos,Pra1996}. These results suffice for our purposes. More general conditions (allowing for jumps) were obtained by Graham \cite{graham1992}.

Let us now give a PDE characterization of $Q_*$. To start, define for each $\omega\in\R^d$:
\[\beta^{\omega}(x) := \int_{\R\times \R^d}\overline{\phi}(x,y,\omega,\pi)\,\Pi_tQ_*(dy,d\pi) + \psi(x,\omega).\]
Also define the integro-differential operator
\[(\mathcal{L}^\omega\,h):=  - \partial_x\,(\beta^\omega\,h)\, + \frac{1}{2}\partial^2_x\,h.\]
Then for each $0\leq t\leq T$, the measure $\Pi_tQ$ has disintegration 
\[\Pi_tQ(dx,d\omega) = \mu(d\omega)\,q^{\omega}_t(dx)\]
where $q^{\omega}$ is a weak solution of the PDE 
\[\partial_t q_t^\omega = \mathcal{L}^\omega\,q^\omega_t.\]
Note that we can rewrite $\beta^\omega$ as
\[\beta^{\omega}(x) := \int_{\R\times \R^d}\overline{\phi}(x,y,\omega,\pi)\,q^{\pi}_t(y)\,\mu(d\pi)\,dy,\]
which obviates the fact that the $q_t^\omega$ for different $\omega$ are coupled.

\subsection{Main theorem} 
We can finally state our Theorem.
\begin{theorem}[Main theorem; proven in subsection \ref{sec:proofmain}]\label{thm:main} Given interaction functions $\psi,\phi,W$ as in Assumption \ref{ass:fg} and \ref{ass:edge}, and measures $\lambda$ (for initial conditions) and $\mu$ (for media variables) as in Assumption \ref{ass:medvartail}, and using Definition \ref{def:rate}:
\begin{enumerate}
\item {\em Existence and uniqueness for McKean-Vlasov problem:} there exists a unique probability measure $Q_*\in \mathcal{M}_1(C([0,T],\R)\times \R^d)$ that is a McKean-Vlasov diffusion for this choice of $\psi,\phi,W,\lambda,\mu$.
\item {\em Large Deviations Principle for $L_n$ and $\overline{L}_n$: } $\{L_n\}_{n\in\N}$ and $\{\overline{L}_n\}_{n\in\N}$ satisfy the same large deviations principle with speed $n$ and rate function $I$. In particular, since $Q_*$ is the only zero of $I$, 
\[\overline{L}_n\mbox{ and }L_n\mbox{ almost surely converge weakly to } Q_*\mbox{ as }n\to +\infty.\]\end{enumerate}\end{theorem}

We emphasize that existence, uniqueness and the LDP for $\overline{L}_n$ come from \cite{Pra1996} with the slight extension discussed in Remark \ref{rem:LDPdifferences} and Appendix \ref{sec:extendingg}. Our new result is that their LDP can be extended to sparse random interactions. 

\subsection{Proof of the main theorem}\label{sec:proofmain}

We now present the proof of Theorem \ref{thm:main}. In fact, most of the actual content of the argument is left for later sections, most notably Section \ref{sec:expequiv}. Our argument consists of two main steps.\\

\noindent {\bf Step 1:} {\em existence and uniqueness for the McKean-Vlasov diffusion and the LDP for $\overline{L}_n$ hold under the assumptions of Theorem \ref{thm:main}.} \\

As noted above, and also in Remark \ref{rem:LDPdifferences}, this essentially follows from a minor modification of the result of \cite{Pra1996}, which we discuss in Appendix \ref{sec:extendingg}.\\

\noindent{\bf Step 2:} {\em transfer the LDP to the sparse random setting.}\\

This is our key contribution. We will need the concept of {\em exponential equivalence}. 

\begin{definition}Let $(S,d)$ be a Polish space. Consider two sequences $\{X_n\}_{n\in\N}$, $\{Y_n\}_{n\in\N}$ of random elements of $S$, with each pair $X_n,Y_n$ defined on the same probability space. We say that the two sequences are {\em exponentially equivalent} if
\[\forall \eta>0\,:\,\limsup_{n\to+\infty}\frac{1}{n}\log \Pr{d(X_n,Y_n)>\eta}=-\infty.\]\end{definition}

For our purposes the key property we will need is the following result.

\begin{lemma}[Version of Lemma 3.13 in \cite{feng2006large}] \label{lem:expequivuse}Let $(S,d)$ be a Polish space. Consider two exponentially equivalent sequences $\{X_n\}_{n\in\N}$, $\{Y_n\}_{n\in\N}$ of random elements of $S$. Assume $\{Y_n\}_{n\in\N}$ satisfies a Large Deviations Principle with good rate function $I$ (cf. Definition \ref{def:LDP}). Then $\{X_n\}_{n\in\N}$ also satisfies a Large Deviations Principle with good rate function $I$\end{lemma}

Recall from Definition \ref{def:process} that \[L_n,\overline{L}_n\in \mathcal{M}_1(C([0,T],\R)\times \R^d).\] In Section \ref{sec:prelim} we noted that that weak convergence in this space is metrized by the bounded Lipschitz distance $d_{BL}$. From Step 1 we know that $\{\overline{L}_n\}_{n\in\N}$ satisfies the LDP we want to prove for $\{L_n\}_{n\in\N}$. So all that we need to prove Theorem \ref{thm:main} is to show that $\{L_n\}_{n\in\N}$ and $\{\overline{L}_n\}_{n\in\N}$ are exponentially equivalent elements of $(\mathcal{M}_1(C([0,T],\R)\times \R^d),d_{BL})$. We do this in Theorem \ref{theo:expequigenphipsi} in Section \ref{sec:expequiv}.

\begin{remark}We emphasize that Theorem \ref{theo:expequigenphipsi} on exponential equivalence requires weaker assumptions that the ``Hamiltonian interactions" in Assumption \ref{ass:fg}. Thus a proof of the LDP for a broader class of interacting diffusions would lead to a generalization of our main result. See also Remark \ref{rem:whyassumehamiltonian} above. \end{remark}
 
\section{Exponential equivalence}\label{sec:expequiv}

In this section we present the main new technical statement in the paper.  

\begin{theorem}\label{theo:expequigenphipsi}\Gnote{le:expequigenphipsi} Consider the systems of diffusions in Definition \ref{def:process}, with all the ingredients introduced in Section \ref{sec:models}. Make Assumptions \ref{ass:edge} and \ref{ass:medvartail}, but replace Assumption \ref{ass:fg} by the weaker assumption that $\phi:\R\times \R\times \R^d\times \R^d\to \R$ and $\psi:\R\times \R^d\to \R$ are bounded functions with bounded derivatives. Then $\{L_n\}_{n\in\N}$ and $\{\overline{L}_n\}_{n\in\N}$ are exponentially equivalent random elements of $\mathcal{M}_1(C([0,T],\R)\times \R^d)$ with the $d_{BL}$ metric:
\[\forall \eta>0\,:\,\limsup_{n\to +\infty}\frac{1}{n}\log \Pr{d_{BL}(L_n,\overline{L}_n)>\eta}=-\infty.\]\end{theorem}

As noted above, Theorem \ref{theo:expequigenphipsi} does {\em not} rely on the assumption of Hamiltonian interactions (Assumption \ref{ass:fg}). 

In the remainder of the section, we give the proof of Theorem \ref{theo:expequigenphipsi}. The main body of the proof are given in Subsection \ref{sub:proofexpequiv}. The proofs of three key lemmas are given in the next subsections. 

\subsection{Proof of exponential equivalence: main steps} \label{sub:proofexpequiv}

\subsubsection{Preliminaries on matrices}
We start by defining some useful notation for the matrices we will have to consider. 
\begin{definition}\label{def:matrices}Under the assumptions of Theorem, we define: 
\begin{eqnarray}\label{eq:defP}P^{(n)}&:=&\frac{A^{(n)}}{p(n)n};\\
\label{eq:defPbar} \overline{P}^{(n)}&:=& \frac{1}{n}(W(\omega_i^{(n)},\omega_j^{(n)}))_{i,j\in[n]}; \mbox{ and }\\
\label{eq:defD}D^{(n)}&:=& P^{(n)} - \overline{P}^{(n)}.\end{eqnarray}
\end{definition}

With this notation, we may rewrite our systems of diffusions as follows (cf. Definition \ref{def:process}). The system with random interactions is given by:

\begin{equation}\label{eq:defsystemP}\left\{\begin{array}{lcl}d\theta_i^{(n)}(t) &=& \left(\sum_{j=1}^nP_{i,j}^{(n)}\,\phi(\theta_i^{(n)}(t),\theta_j^{(n)}(t),\omega_i^{(n)},\omega_j^{(n)})\right)\,dt \\ & &+ \psi(\theta_i^{(n)}(t),\omega_i^{(n)})\,dt + dB_i^{(n)}(t) \\ & &  (0\leq t\leq T,\, i\in[n]); \\ \theta^{(n)}(0) &=& \xi^{(n)}.\end{array}\right.\end{equation}To write the system with dense interactions, we use equation (\ref{eq:defphibar}) and note that
\[\overline{\phi}(x,y,\omega,\pi) = W(\omega,\pi)\,\phi(x,y,\omega,\pi),\]
so that
\begin{equation}\label{eq:defsystemPbar}\left\{\begin{array}{lcl}d\overline{\theta}_i^{(n)}(t) &=& \left(\sum_{j=1}^n\overline{P}_{i,j}^{(n)}\phi(\overline{\theta}_i^{(n)}(t),\overline{\theta}_j^{(n)}(t),\omega_i^{(n)},\omega_j^{(n)})\right)\,dt \\ & & + \psi(\overline{\theta}_i^{(n)}(t),\omega_i^{(n)})\,dt + dB_i^{(n)}(t) \\ & &  (0\leq t\leq T,\, i\in[n]); \\ \overline{\theta}^{(n)}(0) &=& \xi^{(n)}.\end{array}\right.\end{equation}
The key point is that the two systems are nearly the same, the only difference being in the matrices $P^{(n)}$ and $\overline{P}^{(n)}$. The next lemma is the only property of these matrices that we will need. Our statement is essentially contained in the proof of \cite[Lemma 4.1]{Guedon2016}. 

\begin{lemma}[Proof in Subsection \ref{sub:proof.onlymatrix}]\label{lem:onlymatrix} Under our assumptions, for any $0<\eta\leq n$,
\[\Pr{\frac{\|D^{(n)}\|_{\infty\to 1}}{n}>\eta}\leq 4^n\,\expp{-\dfrac{\eta^2n^2p(n)}{8 + \frac{4\eta}{3n}}}.\]
In particular, under the assumption $np(n)\to +\infty$, we have that for all fixed $\eta>0$, 
\[\frac{1}{n}\log \Pr{\frac{\|D^{(n)}\|_{\infty\to 1}}{n}>\eta}\stackrel{n\to+\infty}{\longrightarrow}-\infty.\]\end{lemma}
We emphasize that this Lemma would not hold for more stringent norms such as the spectral norm. 

\subsubsection{A restricted class of pairwise interactions}

We now proceed to prove Theorem \ref{theo:expequigenphipsi} for a restricted class of pairwise interaction functions $\phi$. To define it, recall that a finite complex measure over $\R^{2d+2}$ is a set function:
\[m : \{\mbox{Borel subsets of }\R^{2d+2}\}\to \mathbb{C}\]
of the form $m = m_1 - m_2 + \sqrt{-1}\, (m_3-m_4)$ with each $m_i$ a finite, nonnegative, $\sigma$-additive measure over $\R^d$. We let \begin{equation}\label{eq:mTV}\|m\|_{TV}:= m_1 (\R^d) + m_2(\R^d) + m_3(\R^d) + m_4(\R^d)\end{equation}
denote the total mass of $m$. 

\begin{assumption}[$L^1$ Fourier Class]\Gnote{ass:separableinteraction} \label{ass:separableinteraction} Identify $\R\times\R \times \R^d\times \R^d$ with $\R^{2d+2}$, and write elements of this space as $(x,y,\omega,\pi)$ with $x,y\in\R$ and $\omega,\pi\in\R^d$. We say that $\phi:\R^{2d+2}\to \R$ is in the $L^1$ Fourier class if there exists a finite complex measure $m=m_\phi$ over $\R^{2d+2}$ such that, for all $(x,y,\omega,\pi)\in\R^{2d+2}$,
\[\phi(x,y,\omega,\pi) = \int_{\R^{2d+2}}\,\exp(2\pi \sqrt{-1}\prodint{(x,y,\omega,\pi)}{\vec{z}})\,m_\phi(d\vec{z}).\]\end{assumption}
For instance, $\phi$  is $L^1$-Fourier if it is the inverse $L^1$ transform of a function in $L^1(\R^{2d+2})$. Other examples include $\phi((x,y,\omega,\pi)) = \kappa\sin(y-x)$ (the Kuramoto interaction), which can be expressed via a complex measure that is supported on two points:
\[m_{\rm Kuramoto}= \frac{\kappa}{2\sqrt{-1}}\,\left(\delta_{(-1,1,\vec{0}_{\R^d},\vec{0}_{\R^d})} + \delta_{(1,-1,\vec{0}_{\R^d},\vec{0}_{\R^d})}\right).\]

Our main result for this class is the next lemma.

\begin{lemma}[Proof in Subsection \ref{subsec:auxdense:grown}] \label{le:auxdense:comparison}\Gnote{le:auxdense:comparison} In the setting of Theorem \ref{theo:expequigenphipsi}, assume in addition that $\phi$ is in the $L^1$ Fourier Class with corresponding complex measure $m_\phi$. Then, almost surely,
\[d_W({L_{n}},{\overline{L}_{n}})\leq T\exp\left\{\|W\|_{\infty}\,(2\|\phi\|_{Lip} + \|\psi\|_{Lip})\,T\right\} \frac{4\|m_\phi\|_{TV}\|D^{(n)}\|_{\infty\to 1}}{n}.\]
\end{lemma}

In particular, we may combine this result with Lemma \ref{lem:onlymatrix} to obtain the following statement. 

\begin{corollary}[Proof omitted]\label{cor:expequivL1}In the setting of Theorem \ref{theo:expequigenphipsi}, if $\phi$ is in the $L^1$ Fourier Class, then $L_n$ and $\overline{L}_n$ are exponentially equivalent in the $d_W$ metric:
\[\forall \eta>0\,:\,\frac{1}{n}\log\Pr{d_W({L_{n}},{\overline{L}_{n}})>\eta}\stackrel{n\to +\infty}{\longrightarrow}-\infty.\]\end{corollary}

\subsubsection{Proof for general interactions}\label{sub:proofgeneral}

To finish the proof of Theorem \ref{theo:expequigenphipsi}, we need to extend exponential equivalence to all $\phi$ that are bounded and have bounded derivative. 

\begin{equation}\label{eq:goalexpequiv}  \left.\begin{array}{lll}\mbox{{\bf Goal (1): } for any $\eta>0$,}& & \\ \frac{1}{n}\log\Pr{d_{BL}(L_{n},\overline{L}_{n})>\eta}\stackrel{n\to +\infty}{\longrightarrow}-\infty.\end{array}\right.\end{equation}

This will follow from an approximation argument. We need the following technical result. 

\begin{lemma}[Good Approximation; proven in Appendix \ref{subsec:app}] \label{lemma:goodapp} Let \[\phi:\R^{2d+2}\to\R\] be a bounded function with bounded derivative. Then there exists a family of functions $(\phi^{\varepsilon,R})_{R\geq 1,0<\eps\leq 1}$ that give a good approximation for $\phi$ in the following sense: there exists a constant $M=M(\supnorm{\phi},\pnorm{\infty}{\nabla \phi})$ independent of $\varepsilon$ and $R$ such that:
\begin{enumerate}
\item $\phi^{\eps,R}$ is smooth with compact support.
\item \label{uniformC1} the $\phi^{\eps,R}$ have uniformly bounded $C^1$ norm: $\supnorm{\phi^{\varepsilon,R}} + \pnorm{\infty}{\nabla\phi^{\varepsilon,R}}\leq M$.
\item \label{goodinside} $\phi^{\varepsilon,R}$ approximates $\phi$ in $B_R(\vec{0})$ in the sense that $$\pnorm{L^{\infty}(B_R(\vec{0}))}{\phi^{\varepsilon,R}-\phi}\leq \varepsilon.$$
\end{enumerate}
\end{lemma}

The importance of this result is that each $\phi^{\eps,R}$ is in the Schwarz class and is thus $L^1$ Fourier. 

To continue the proof, we take a family of good approximations as in Lemma \ref{lemma:goodapp} for the $\phi$ in the definition of our diffusion. For each $R>0$ and $\varepsilon>0$ as above, we let 
\[L^{\eps,R}_n \mbox{ and }\overline{L}_n^{\eps,R}\]
be the empirical measures of the two processes $L_n$ and $\overline{L}_n$ when $\phi$ is replaced with $\phi^{\eps,R}$ in equations (\ref{eq:defsystemP}) and (\ref{eq:defsystemPbar}). By the triangle inequality, 

\begin{eqnarray} d_{BL}(L_{n},\overline{L}_{n}) &\leq & d_{BL}(L^{\eps,R}_{n},\overline{L}^{\eps,R}_{n})\nonumber\\ & & + d_{BL}(L_{n},L^{\eps,R}_{n}) + d_{BL}(\overline{L}_{n},\overline{L}^{\eps,R}_{n}).\end{eqnarray}

Corollary \ref{cor:expequivL1} applies to the first term since $\phi^{\eps,R}$ is $L^1$ Fourier. The bound $d_{BL}\leq d_W$ then implies:
\[\forall \eta>0\,:\,\frac{1}{n}\log\Pr{d_{BL}(L^{\varepsilon,R}_{n},\overline{L}^{\varepsilon,R}_{n})>\eta}\stackrel{n\to +\infty}{\longrightarrow} -\infty.\]

So all that is left to show is that:
\begin{equation}\label{eq:goalexpequiv2}  \left.\begin{array}{lll}\mbox{{\bf Goal (2): } for any $\eta>0$,}& & \\ \inf_{\varepsilon,R}\limsup_{n\to +\infty}\frac{1}{n}\log\Pr{d_W(L_{n},L^{\varepsilon,R}_{n})>\eta}&=&-\infty;\\ \inf_{\varepsilon,R}\limsup_{n\to +\infty}\frac{1}{n}\log \Pr{d_W(\overline{L}_{n},\overline{L}^{\varepsilon,R}_{n})>\eta}&=&-\infty.\end{array}\right.\end{equation}
Notice that the infima above are over parameters $\varepsilon\in (0,1]$ and $R\geq 1$. 

The next Lemma describes the effect on the empirical measure of replacing $\phi$ with $\phi^{\eps,R}$ in the interactions.

\begin{lemma}[Proof in Subsection \ref{subsec:proof:le:rapp}]\label{le:Rapp}\Gnote{le:Rapp} With the above assumptions and notation, define the bad events
\[E^{(n)}_{i}(R):= \left\{\sup_{s \in [0,T] }|\theta^{(n)}_i(s)|>R/4 \ \mbox{or} \  \ |\omega_i^{(n)}|>R/4\right\} \; (i \in [n])\]
and
\[\overline{E}^{(n)}_{i}(R):= \left\{\sup_{s \in [0,T] }|\overline{\theta}^{(n)}_i(s)|> R/4 \ \mbox{or} \  \ |\omega^{(n)}_i|>R/4\right\} \; (i \in [n])\]

Then there exists a positive constant $C>0$ depending only on $\psi$, $\phi$ such that the following is a.s. true:
\[d_{BL}(L_n,L^{\eps,R}_n)\leq CT\,\exp(C\|W\|_\infty T)\,\left(\eps + \frac{\|D\|_{\infty\to 1}}{n} + \frac{\sum_{i=1}^n\Ind{E_i(R)}}{n}\right)\]
and similarly for $d_{BL}(\overline{L}^{\eps,R}_{n},\overline{L}_{n})$ with $\overline{E}^{(n)}_v(R)$ replacing $E^{(n)}_v(R)$.\end{lemma}

An important point in this Lemma is that neither the ``bad events" nor the constant $C$ depend on $\eps$. 

We now apply Lemma \ref{le:Rapp}. Recall from Lemma \ref{lem:onlymatrix} that $\|D\|_{\infty\to 1}/n$ is exponentially equivalent to $0$. Also, the $\eps$ appearing in that expression can be taken arbitrarily small. Comparing Goal (2) with the expression in Lemma \ref{le:Rapp}, we see that it suffices to achieve the following.

\begin{equation}\label{eq:goalexpequiv3}  \left.\begin{array}{lll}\mbox{{\bf Goal (3): } for any $\eta>0$,}& & \\ \inf_{R}\limsup_{n\to +\infty}\frac{1}{n}\log\Pr{\frac{1}{n}\sum_{v=1}^n\Ind{E^{(n)}_v(R)}>\eta}&=&-\infty;\\ \inf_{R}\limsup_{n\to +\infty}\frac{1}{n}\log \Pr{\frac{1}{n}\sum_{v=1}^n\Ind{\overline{E}^{(n)}_v(R)}>\eta}&=&-\infty.\end{array}\right.\end{equation}
This last goal essentially consists of controlling the probability that many media variables and/or many diffusions fall outside of a large ball. For this we use the next Lemma.

\begin{lemma}[Proof in Subsection \ref{subsec:proof:le:exitcompact}]\label{le:exitcompact}\Gnote{le:exitcompact} Under the assumptions of Theorem \ref{theo:expequigenphipsi}, there exist numbers $a_T(\eta,R)$ depending on $\phi$, $\psi$, $\eta$, $T$, $R$, $\lambda$, and $\mu$ such that $a_T(\eta,R)\to +\infty$ when $R\to +\infty$,
\[\limsup_{n\in\N}\frac{1}{n}\log\Pr{\frac{1}{n}\sum_{i=1}^n\Ind{E^{(n)}_i(R)}>\eta}\leq -a_T(\eta,R),\]
and
\[\limsup_{n\in\N}\frac{1}{n}\log\Pr{\frac{1}{n}\sum_{i=1}^n\Ind{\overline{E}^{(n)}_i(R)}>\eta}\leq -a_T(\eta,R).\]\end{lemma}

Applying this Lemma implies Goal (3) in \ref{eq:goalexpequiv3} and finishes the proof of Theorem \ref{theo:expequigenphipsi} (modulo the Lemmas proven below). 

\subsection{Exponetial equivalence for a class of interactions}\label{subsec:auxdense:grown}
We now prove Lemma \ref{le:auxdense:comparison}. This Lemma says, when $\phi$ belongs to the $L^1$ Fourier Class, we can bound the distance between $L_{n}$ and $\overline{L}_{n}$ in terms of the difference of the matrices $P$ and $\overline{P}$ appearing in Definition \ref{def:matrices}. To avoid cumbersome notation, we will omit the superscripts $(n)$ from all objects. 

\subsubsection{Preliminaries}\label{sub:auxdense:L1fourier}
For $i\in[n]$ and $0\leq t\leq T$, define:
\begin{eqnarray}\delta_i(t)&:=&\theta_i(t) - \overline{\theta}_i(t);\\
\pnorm{\infty,t}{\delta_i}&:=&\sup_{0\leq s\leq t}|\delta_i(s)|;\\
\Delta(t) &:=& \frac{1}{n}\sum_{i=1}^n\pnorm{\infty,t}{\delta_i}.\end{eqnarray}
For any $1$-Lipschitz function $h:C([0,T],\R)\times \R^d\to \R$, 
\begin{eqnarray}\nonumber \int h(\theta,\omega)\,(L_{n}-\overline{L}_{n})\,(d\theta,d\omega) & = &  \dfrac{1}{n}\sum_{i=1}^{n}(h(\overline{\theta}_{i},\omega_v)-h(\theta_{i},\omega_i))\\ \label{eq:boundviacoupling} &\leq & \frac{1}{n}\sum_{i=1}^n\pnorm{\infty,T}{\delta_i} = \Delta(T).
\end{eqnarray}
Note that $\Delta(0)=0$ and $\Delta(t)$ is continuous in $t$. In subsection \ref{sub:effectsingle} we derive expressions for the $\delta_i(t)$ (Proposition \ref{prop:singleL1}). In subsection \ref{sub:gronwall1} we use these expressions and a Gronwall-style argument to show:
\begin{equation}\label{eq:goaldelta(t)} \Delta(T)\leq T\exp\left\{\|W\|_{\infty}\,(2\|\phi\|_{Lip} + \|\psi\|_{Lip})\,T\right\} \frac{4\|m_\phi\|_{TV}\|D\|_{\infty\to 1}}{n}.\end{equation}
The Lemma then follows from a direct plugin into (\ref{eq:boundviacoupling}).

\subsubsection{The difference between trajectories for a single unit}\label{sub:effectsingle}
\begin{proposition}\label{prop:singleL1}For each $i\in[n]$ and $0\leq t\leq T$, we have the following formula for $\delta_i(t) = \theta_i(t) - \overline{\theta}_i(t)$:
\begin{eqnarray}\label{eq:firstline2}\delta_i(t) &=&  \int_0^t\,\left(\sum_{j=1}^n\overline{P}_{i,j}\Phi_{i,j}(s) + \Psi_i(s)\right)\,ds \\
\label{eq:secondline2} & & +\int\limits_{\R^{2d+2}}\int_0^t\,\sum_{j=1}^nD_{i,j}\,a_i(\vec{z},s)\,b_j(\vec{z},s)\,ds\,m_\phi(d\vec{z}), \end{eqnarray}
with $\overline{P}$ and $D$ as in Definition \ref{def:matrices}, \begin{eqnarray}\label{eq:boundPhi}|\Phi_{i,j}(s)|&\leq & \|\phi\|_{Lip}\,(|\delta_i(s)| + |\delta_j(s)|);\\
\label{eq:boundPsi} |\Psi_i(s)|&\leq & \|\psi\|_{Lip}\,|\delta_i(s)|;\end{eqnarray}
$m_\phi$ the complex measure associated with $\phi$, and $a_i,b_i:\R^{2d+2}\times [0,T]\to \mathbb{C}$ are bounded by $1$ in absolute value.\end{proposition}
\begin{proof}Recall that in this subsection we omit the $(n)$ superscript. Define:
\begin{eqnarray*}\Phi_{i,j}(s)&:=& \phi(\theta_i(s),\theta_j(s),\omega_i,\omega_j) \\ \noindent & & - \phi(\overline{\theta}_i(s),\overline{\theta}_j(s),\omega_i,\omega_j);\\
\Psi_i(s)&:=& \psi(\theta_i(s),\omega_i) - \psi(\overline{\theta}_i(s),\omega_i).\end{eqnarray*}
Direct comparison with equations (\ref{eq:defsystemP}) and (\ref{eq:defsystemPbar}) and the definition $D=P-\overline{P}$ in Definition \ref{def:matrices} give:
\begin{eqnarray}\label{eq:firstline}\delta_i(t) &=&  \int_0^t\,\left(\sum_{j=1}^n\overline{P}_{i,j}\Phi_{i,j}(s) + \Psi_i(s)\right)\,ds \\
\label{eq:secondline} & & +\int_0^t\,\left(\sum_{j=1}^nD_{i,j}\,\phi(\theta_i(s),\theta_j(s),\omega_i,\omega_j)\right)\,ds.\end{eqnarray}
Properties (\ref{eq:boundPhi}) and (\ref{eq:boundPsi}) follow from the Lipschitz assumptions in Theorem \ref{theo:expequigenphipsi}.

To finish the proof, we apply the assumption that $\phi$ is $L^1$ Fourier (Assumption \ref{ass:separableinteraction}). Writing $\vec{z}\in\R^d$ as \[\vec{z}=(z_1,z_2,z_3,z_4)\mbox{ with }z_1,z_2\in\R\mbox{ and }z_3,z_4\in\R^d,\]
we obtain
\begin{equation}\label{eq:phiL1}\phi(\theta_i(s),\theta_j(s),\omega_i,\omega_j) = \int_{\R^{2d+2}}\,a_i(\vec{z},s)\,b_j(\vec{z},s)\,m_\phi(d\vec{z})\end{equation}
for $m$ a finite complex-valued measure and functions $a_i$, $b_j$ defined as follows:
\begin{eqnarray}\label{eq:defai} a_i(\vec{z},s)&:=&\exp(2\pi\sqrt{-1}\,(\theta_i(s)z_1 + \prodint{\omega_i}{z_3}));\\
\label{eq:defbj} b_j(\vec{z},s) &:=& \exp(2\pi\sqrt{-1}\,(\theta_j(s)z_2 + \prodint{\omega_j}{z_4})).\end{eqnarray}
Integrating (\ref{eq:phiL1}) from $s=0$ to $t$ then finishes the proof of Proposition \ref{prop:singleL1}.\end{proof}

\subsubsection{The Gronwall argument}\label{sub:gronwall1}

\begin{proposition}\label{prop:gronwallL1}Let $\Delta(t)$ be as in Subsection \ref{sub:auxdense:L1fourier}. Then:
\[\Delta(T)\leq T\exp\left\{\|W\|_{\infty}\,(2\|\phi\|_{Lip} + \|\psi\|_{Lip})\,T\right\} \frac{4\|m_\phi\|_{TV}\|D\|_{\infty\to 1}}{n}.\]\end{proposition}
\begin{proof}We use the formulae in Proposition \ref{prop:singleL1} to derive a Gronwall-style bound $\Delta(t)$.  For each $i\in[n]$, we may choose $u_i\in \{-1,+1\}$ and $0\leq t_i\leq t$ so that:
\[\pnorm{\infty,t}{\delta_i} = u_i\,\delta_i(t_i).\]Then Proposition \ref{prop:singleL1} implies:
\begin{eqnarray}\nonumber \Delta(t) &=& \frac{1}{n}\sum_{i=1}^n\,u_i\,\delta_i(t_i)\\
\label{eq:firstline3}&=& \int_0^{t}\left(\sum_{i,j=1}^n\,\frac{u_i\overline{P}_{i,j}}{n}\Phi_{i,j}(s)\Ind{[0,t_i]}(s) + \sum_{i=1}^n\frac{\Psi_i(s)}{n}\Ind{[0,t_i]}(s)\right)ds \\
\label{eq:secondline3} & & +\int\limits_{\R^{2d+2}}\int_0^t\,\sum_{i,j=1}^n\frac{D_{i,j}}{n}\,(u_ia_i(\vec{z},s)\Ind{[0,t_i]}(s))\,b_j(\vec{z},s)\,ds\,m_\phi(d\vec{z}).\end{eqnarray}
The term in (\ref{eq:firstline3}) can be bounded using (\ref{eq:boundPhi}), (\ref{eq:boundPsi}) from Proposition \ref{prop:singleL1} in combination with $\overline{P}_{i,j}\leq \|W\|_{\infty}/n$. We obtain:
\begin{eqnarray}\nonumber \mbox{Term in (\ref{eq:firstline3})}&\leq & \|W\|_{\infty}\,(2\|\phi\|_{Lip} + \|\psi\|_{Lip})\,\int_0^t\,\sum_{i=1}^n\frac{|\delta_i(s)|}{n}\Ind{[0,t_i]}(s)\,ds\\ \label{eq:firstline4} &\leq & \|W\|_{\infty}\,(2\|\phi\|_{Lip} + \|\psi\|_{Lip})\,\int_0^t\Delta(s)\,ds.\end{eqnarray}
As for the RHS of (\ref{eq:secondline3}), it has the form
\[\int\limits_{\R^{2d+2}}\int_0^t\,\prodint{\vec{\tilde{a}}(\vec{z},s)}{D\vec{b}(\vec{z},s)}\,ds\,m_\phi(d\vec{z})\]
where 
\[\vec{\tilde{a}}(\vec{z},s) = (u_i\,\Ind{[0,t_i]}(s)\,a_i(\vec{z},s))_{i\in[n]}\mbox{ and }\vec{b}(\vec{z},s) = (b_j(\vec{z},s))_{j\in[n]}.\]
It follows from the properties of $a_i$ and $b_j$ in Proposition \ref{prop:singleL1} that the vectors $\vec{\tilde{a}}(\vec{z},s)$ and $\vec{b}(\vec{z},s)$ are complex vectors with $\ell^\infty$ norms bounded by $1$.  Decomposing each vector into real and complex parts, we see that:
\[\prodint{\vec{\tilde{a}}(\vec{z},s)}{D\vec{b}(\vec{z},s)}\leq 4\,\sup\{\prodint{\vec{x}}{D\vec{y}}\,:\, \vec{x},\vec{y}\in [-1,1]^n\} = 4\|D\|_{\infty\to 1}.\]
Plugging this in (\ref{eq:secondline3}) and also (\ref{eq:firstline4}) into (\ref{eq:firstline3}), we obtain:
\[ \Delta(t)\leq \|W\|_{\infty}\,(2\|\phi\|_{Lip} + \|\psi\|_{Lip})\,\int_0^t\Delta(s)\,ds + \frac{4\|m_\phi\|_{TV}\,t\,\|D\|_{\infty\to 1}}{n}.\]
Gronwall's inequality then gives:
\[\Delta(T) \leq T\exp\left\{\|W\|_{\infty}\,(2\|\phi\|_{Lip} + \|\psi\|_{Lip})\,T\right\} \frac{4\|m_\phi\|_{TV}\,\|D\|_{\infty\to 1}}{n},\]
which is the desired inequality.\end{proof}

\subsection{On the approximation of interaction functions}\label{subsec:proof:le:rapp}

We now prove Lemma \ref{le:Rapp}, which quantifies the effect of replacing function $\phi$ in Definition \ref{def:process} with a good approximation $\phi^{\eps,R}$ as in Lemma \ref{le:goodapp}. For simplicity, we only present in detail the part of the argument where $L_n$ and $L_n^{\eps,R}$ are compared. The comparison of $\overline{L}_n$ and $\overline{L}_n^{\eps,R}$ is similar (in fact simpler). 

\subsubsection{Preliminaries} 

For the remainder of the section, we mostly omit superscripts $(n)$ from our notation. Parameters $R\geq 1$ and $\eps\in (0,1]$ are fixed from now on. 

We begin by writing down the system of diffusions for $L^{\eps,R}_n$. As explained in Subsection \ref{sub:proofgeneral}, this system is obtained from (\ref{eq:defsystemP}) by replacing $\phi$ with its approximation $\phi^{\eps,R}$. That is, the corresponding diffusions 
\[\theta^{\eps,R}:=(\theta_i^{\eps,R})_{i\in[n]}\in C([0,T],\R)^n\]
satisfy:
\begin{equation}\label{eq:driftepsR}\left\{\begin{array}{lcl}d\theta_i^{\eps,R}(t) &=& \left(\sum_{j=1}^nP_{i,j},\phi^{\eps,R}(\theta^{\eps,R}_i(t),\theta^{\eps,R}_j(t),\omega_i,\omega_j)\right)\,dt \\ & &+ \psi(\theta_i^{\eps,R}(t),\omega_i^{(n)})\,dt + dB_i^{(n)}(t) \\ & &  (0\leq t\leq T,\, i\in[n]); \\ \theta^{\eps,R}(0) &=& \theta(0) = \xi;\end{array}\right.\end{equation}

As in subsection \ref{sub:auxdense:L1fourier}, we control the difference between $L_n$ and $L^{\eps,R}_n$ via pairwise comparison of the trajectories. However, in this case we use the less stringent BL norm instead of the Wasserstein metric. Thus, if we define
\[\delta^{\eps,R}_i(t):= \theta_i(t) - \theta^{\eps,R}_i(t),\,\,(0\leq t\leq T, i\in[n])\]
and set
\[\pnorm{\infty,t}{\delta^{\eps,R}_i}:=\sup_{0\leq s\leq t}|\delta^{\eps,R}_i(s)|,\]
we note that, for any $h:C([0,T],\R)\times \R^d\to \R$ with $\|h\|_{BL}\leq 1$ (cf. Remark \ref{rem:bldist}), 
\begin{eqnarray}\nonumber \int h(\theta,\omega)\,(L_{n}-\overline{L}_{n})\,(d\theta,d\omega) & = &  \dfrac{1}{n}\sum_{i=1}^{n}(h(\theta_{i},\omega_i)-h(\overline{\theta}_{i},\omega_i))\\ \nonumber &\leq & \frac{1}{n}\sum_{i=1}^n\pnorm{\infty,T}{\delta^{\eps,R}_i}\wedge 1,\end{eqnarray}
which implies
\begin{equation} \label{eq:boundepsRfinal} d_{BL}(L_n,L_n^{\eps,R}) \leq  \widetilde{\Delta}^{\eps,R}(T):=\frac{1}{n}\sum_{i=1}^n\pnorm{\infty,T}{\delta^{\eps,R}_i}\wedge 1.\end{equation}

As in the proof of Lemma \ref{le:auxdense:comparison} in Subsection \ref{sub:auxdense:L1fourier}, we will  apply Gronwall's inequality to bound $\widetilde{\Delta}^{\eps,R}(T)$. To start, we write in Subsection \ref{sub:singleRapp} a formula for $\delta^{\eps,R}_i(t)$ for a single $i\in[n]$. This formula is then applied in Subsection \ref{sub:gronwallRapp} to prove that:
\[\widetilde{\Delta}^{\eps,R}(T)\leq CT\,\exp(C\|W\|_\infty\,T)\,\left(\eps + \frac{\|D\|_{\infty\to 1}}{n} + \frac{\sum_{i=1}^n\Ind{E_i(R)}}{n}\right)\]
for some $C$ depending on $\psi$ and $\phi$ only. This proves Lemma \ref{le:Rapp} via a direct plugin into (\ref{eq:boundepsRfinal}). 

\begin{remark}In this proof we use the bounded Lipschitz metric instead of the stronger $d_W$ (contrast with Lemma \ref{le:auxdense:comparison}). The main reason we do this is to facilitate later use of the $\|\cdot\|_{\infty\to 1}$ norm. See Remark \ref{rem:inftyone} below for more details.\end{remark}

\subsubsection{The difference between trajectories for one unit}\label{sub:singleRapp}
\begin{proposition}\label{prop:singleRapp}For each $i\in[n]$ and $0\leq t\leq T$, 
\begin{eqnarray} \nonumber \delta^{\eps,R}_i(t) &=&  \int_0^t \left(\sum_{j=1}^n\,P_{i,j}\alpha^{\eps,R}_{i,j}(s)\right)\,ds \\ \label{eq:deltaepsRbound2} & & +\int_0^t \left(\sum_{j=1}^n\,P_{i,j}\Phi^{\eps,R}_{i,j}(s) + \Psi^{\eps,R}_i(s)\right)\,ds,\end{eqnarray}
and there exists $C>0$ depending only on $\phi$ (and not on $\eps,R$) such that: 
\begin{eqnarray}\label{eq:boundPhiepsR}|\Phi^{\eps,R}_{i,j}(s)|&\leq&  C\,(|\delta_i^{\eps,R}(s)|\wedge 1 +|\delta_j^{\eps,R}(s)|\wedge 1),\\
\label{eq:boundPsiepsR} |\Psi^{\eps,R}_{i}(s)| &\leq & C\,|\delta_i^{\eps,R}(s)|\wedge 1,\mbox{ and }\\
\label{eq:boundalphaepsR}|\alpha^{\eps,R}_{i,j}(s)| &\leq&C\,(\eps + \Ind{E_i(R)} + \Ind{E_j(R)}).\end{eqnarray}
\end{proposition}
\begin{proof}The notion of good approximation in Lemma \ref{lemma:goodapp} guarantees that the functions $\phi^{\eps,R}$ have $\|\phi^{\eps,R}\|_{\infty}+ \|\nabla\phi^{\eps,R}\|_{{\rm op},\infty}\leq M$. This implies in particular that $\phi^{\eps,R}$ is $M$-Lipschitz and bounded by $M$ for all choices of $\eps$ and $R$ and leads to (\ref{eq:boundPhiepsR}). This will be used below. 

We now define:
\begin{eqnarray}\label{eq:defalphaepsR}\alpha^{\eps,R}_{i,j}(s)&:=& \phi(\theta_i(s),\theta_j(s),\omega_i,\omega_j) \\ \nonumber & & - \phi^{\eps,R}(\theta_i(s),\theta_j(s),\omega_i,\omega_j);\\ \nonumber 
\Phi^{\eps,R}_{i,j}(s) &:=& \phi^{\eps,R}(\theta_i(s),\theta_j(s),\omega_i,\omega_j) \\ \nonumber & & - \phi^{\eps,R}(\theta^{\eps,R}_i(s),\theta^{\eps,R}_j(s),\omega_i,\omega_j)\\ \nonumber 
\Psi^{\eps,R}_i(s)&:=& \psi(\theta_i(s),\omega_i) -  \psi(\theta^{\eps,R}_i(s),\omega_i),\end{eqnarray}
Direct comparison of the drift terms of $\theta_i(t)$ (in equation (\ref{eq:defsystemP})) and $\theta_i^{\eps,R}$ (in (\ref{eq:driftepsR})) gives (\ref{eq:deltaepsRbound2}). 

The bound in (\ref{eq:boundPhiepsR}) follows from the fact that $\phi^{\eps,R}$ is $M$-Lipschitz and bounded by $M$. The argument for (\ref{eq:boundPsiepsR}) relies on the Lipschitz constant and boundedness of $\psi$. 

To prove the bound in (\ref{eq:boundalphaepsR}) for $\alpha^{\eps,R}(s)$, we go back to the definition (\ref{eq:defalphaepsR}). Using again Lemma \ref{lemma:goodapp} we note that
\[|(\theta_i(s),\theta_j(s),\omega_i,\omega_j)|\leq R\Rightarrow |\alpha^{\eps,R}(s)|\leq \eps.\]
Now recall that
\[E^{(n)}_{i}(R):= \left\{\sup_{s \in [0,T] }|\theta^{(n)}_i(s)|\geq R/4 \ \mbox{or} \  \ |\omega_i^{(n)}|>R/4\right\} \; (i \in [n]).\]
In particular, when $|(\theta_i(s),\theta_j(s),\omega_i,\omega_j)|> R$, at least one of the indicators $\Ind{E_i(R)}$ or $\Ind{E_j(R)}$ is $1$, and we still have the bound $|\alpha^{\eps,R}(s)|\leq 2M$. We deduce:
\[|\alpha^{\eps,R}(s)|\leq (\eps + 2M\,(\Ind{E_i(R)} + \Ind{E_j(R)})),\]
as desired.\end{proof}

\subsubsection{The Gronwall argument}\label{sub:gronwallRapp}

\begin{proposition}\label{prop:gronwallRapp}For $0\leq t\leq T$, let
\[\widetilde{\Delta}^{\eps,R}(t):= \frac{1}{n}\sum_{i=1}^n\pnorm{\infty,t}{\delta_i^{\eps,R}}\wedge 1.\]
Then
\[\widetilde{\Delta}^{\eps,R}(T)\leq CT\,\exp(C\|W\|_\infty\,T)\,\left(\eps + \frac{\|D\|_{\infty\to 1}}{n} + \frac{\sum_{i=1}^n\Ind{E_i(R)}}{n}\right)\]
for some $C>0$ that only depends on $\psi$ and $\phi$.\end{proposition} 

\begin{proof}We may apply Proposition \ref{prop:singleRapp} using the bounds for $\Phi^{\eps,R}_{i,j}(s)$, $\Psi^{\eps,R}_i(s)$ and $\alpha^{\eps,R}_{i,j}(s)$ and deduce that, for any $0\leq t\leq T$
\begin{eqnarray*} \nonumber \pnorm{\infty,t}{\delta^{\eps,R}_i}\wedge 1&\leq &  C\,\int_0^t \sum_{j=1}^nP_{i,j}\left(\pnorm{\infty,s}{\delta^{\eps,R}_i}\wedge 1  + \pnorm{\infty,s}{\delta^{\eps,R}_j}\wedge 1\right)\,ds\\ & & + Ct\,\sum_{j=1}^nP_{i,j}\,(\eps + \Ind{E_i(R)} + \Ind{E_j(R)}).\end{eqnarray*}

If we average these expression over $i\in[n]$, the RHS becomes $\widetilde{\Delta}^{\eps,R}(t)$, and we obtain:
\begin{equation}\label{eq:wildetildedelta1}\widetilde{\Delta}^{\eps,R}(t) \leq C\,\sum_{i=1}^n\frac{S_{i}}{n}\left(\int_0^t\pnorm{\infty,s}{\delta^{\eps,R}_i}\wedge 1\,ds   + t\eps + t\Ind{E_i(R)}\right),\end{equation}
where \begin{equation}\label{eq:defsi}S_i:=\sum_{j=1}^n(P_{i,j}+P_{j,i}) = 2\sum_{j=1}^nP_{i,j}\,\,(i\in[n]).\end{equation}
(The last equality holds above because $P$ is symmetric.)

We now put (\ref{eq:wildetildedelta1}) in the form of a inner product. Let $\one\in\R^n$ be the vector in $\R^n$ with all cordinates equal to $1$. Also define:
\begin{equation}\label{eq:defv}\vec{v}(t):= \left(\int_0^t\pnorm{\infty,s}{\delta^{\eps,R}_i}\wedge 1\,ds   + t\eps + t\Ind{E_i(R)}\right)_{i\in[n]}\in\R^n.\end{equation}
Then $2P\one = (S_i)_{i\in[n]}$ and (\ref{eq:wildetildedelta1}) can be rewritten as:
\begin{equation}\label{eq:wildetildedelta2}\widetilde{\Delta}^{\eps,R}(t) \leq \frac{2C}{n}\,\prodint{\vec{v}(t)}{P\one}.\end{equation}
Now $\supnorm{\vec{v}(t)}\leq (2+\eps)\,t\leq 3\,t$ if $\eps\leq 1$. We also have $\supnorm{\one}\leq 1$. Recalling $D = P-\overline{P}$ (cf. Definition \ref{def:matrices}), we deduce:
\begin{equation}\label{eq:wildetildedelta3}\widetilde{\Delta}^{\eps,R}(t) \leq \frac{2C}{n}\,\prodint{\vec{v}(t)}{(\overline{P} + D)\one} \leq  \frac{2C}{n}\,\prodint{\vec{v}(t)}{\overline{P}\one}+ \frac{6C\|D\|_{\infty\to 1}\,t}{n}.\end{equation}
The entries of $\overline{P}$ are bounded by $\|W\|_{\infty}/n$, so the coordinates of $\overline{P}\one$ are all bounded by $\|W\|_{\infty}$. We deduce:
\begin{eqnarray*}\widetilde{\Delta}^{\eps,R}(t) -\frac{6C\|D\|_{\infty\to 1}\,t}{n}&\leq & \frac{2C}{n}\,\prodint{\vec{v}(t)}{\overline{P}\one}\\ &\leq& \frac{2C\|W\|_{\infty}}{n}\prodint{\vec{v}(t)}{\one} \\ \mbox{(use defn. of $\vec{v}(t)$, (\ref{eq:defv}))} &\leq & 2C\|W\|_{\infty}\,\frac{1}{n}\int_0^t\sum_{i=1}^n\pnorm{\infty,s}{\delta^{\eps,R}_i}\wedge 1\,ds\\ & &    +{2Ct\eps} + t\frac{\sum_{i=1}^n\Ind{E_i(R)}}{n},\end{eqnarray*}
or more explicitly 
\begin{eqnarray*} \widetilde{\Delta}^{\eps,R}(t) &\leq& 2C\|W\|_{\infty}\,\int_0^t\widetilde{\Delta}^{\eps,R}(s)\,ds + \frac{6C\|D\|_{\infty\to 1}\,t}{n}  \\ & & + {2Ct\eps} + 2Ct\frac{\sum_{i=1}^n\Ind{E_i(R)}}{n}.\end{eqnarray*}
Gronwall's inequality now gives:
\[\widetilde{\Delta}^{\eps,R}(T)\leq 6CT\,\exp(2C\|W\|_\infty\,T)\,\left(\eps + \frac{\|D\|_{\infty\to 1}}{n} + \frac{\sum_{i=1}^n\Ind{E_i(R)}}{n}\right),\]
which is the desired result once we ``redefine $C$ as $6C$".

\begin{remark}\label{rem:inftyone}Note that (\ref{eq:wildetildedelta3}) only ``works" because the coordinates of $\vec{v}(t)$ are bounded. This is a consequence of considering \[\pnorm{\infty,t}{\delta_i^{\eps,R}}\wedge 1\mbox{ instead of }\pnorm{\infty,t}{\delta_i^{\eps,R}}.\] The ultimate reason why we have the $\wedge 1$'s is that we used the $d_{BL}$ metric to compare the empirical measures. This explains why we used this metric instead of $d_W$.  \end{remark}

\end{proof}

\section{Proofs of some additional lemmas}\label{sec:auxiliary}

\subsection{Matrix concentration in the $\|\cdot\|_{\infty\to 1}$ norm}\label{sub:proof.onlymatrix}

We prove here Lemma \ref{lem:onlymatrix}. Recall the definition of the matrices $P^{(n)}$, $\overline{P}^{(n)}$ and $D^{(n)} = P^{(n)} - \overline{P}^{(n)}$ from Definition \ref{def:matrices}. 

\begin{proof}For convenience, we omit the $(n)$ superscripts. Our argument is based on Bennett's concentration inequality:
\begin{lemma}[Bennett's Inequality, \cite{boucheron2013concentration}, Theorem 2.9, section 2.7]\label{bennet} Let	$X_1, \dots, X_k$ be independent random variables with finite variance and $X_i \leq b $ a.s. for a constant $b>0$. Let $S=\sum_{i=1}^{k}(X_i - \Ex{X_i})$ and $v=\sum_{i=1}^{k}\Ex{X_{i}^{2}}$. Then
\begin{equation} \label{eq:berstein}
\Pr{S \geq t} \leq \expp{-\dfrac{t^2}{2v+2/3bt}}. \nonumber
\end{equation}
\end{lemma}
The argument is an easy modification of \cite[Lemma 4.1]{Guedon2016}. 
Recall that:
\[\|D\|_{\infty\to 1} = \sup\{\prodint{\vec{x}}{D\vec{y}}\,:\, \vec{x},\vec{y}\in[-1,1]^n\}.\] 
Since $[-1,1]^n$ is the convex hull of $\{-1,1\}^n$, one can see at once that the supremum in the RHS is achieved at some pair $\vec{x},\vec{y}\in\{-1,+1\}^n$. Since there are $4^n$ such pairs,
\[\Pr{\frac{\|D\|_{\infty\to 1}}{n}>\eta}\leq 4^n\,\max_{\vec{x},\vec{y}\in\{-1,1\}^n}\Pr{\prodint{\vec{x}}{D\vec{y}}>\eta n}.\]
We will be done once we show that 
\[\mbox{\bf Goal: }\max_{\vec{x},\vec{y}\in\{-1,1\}^n}\Pr{\prodint{\vec{x}}{D\vec{y}}>\eta n}\leq \expp{-\dfrac{\eta^2n^2p(n)}{8 + \frac{4\eta}{3n}}},\]
as the exponent in the RHS of this expression grows superlinearly with $n$ (recall $np(n)\to +\infty$). 

Section \ref{sec:models} specifies that, conditionally on specific values of the $\omega_i$, the $A_{i,j}$ with $i\leq j$ are independent Bernoulli random variables with respective means $p(n)W(\omega_i,\omega_j)$. It follows that, for fixed $\vec{x},\vec{y}\in \{-1,1\}^n$, 
\begin{eqnarray*}\prodint{\vec{x}}{D\vec{y}} &=& \sum_{1\leq i<j\leq n}\frac{2x_iy_j}{p(n)n}\,(A_{i,j}-p(n)W(\omega_i,\omega_j)) \\ & & + \sum_{i=1}^n \frac{x_iy_i}{p(n)n}\,(A_{i,i} - p(n)W(\omega_{i},\omega_i))\end{eqnarray*}
is a sum of at most $n^2$ independent mean-$0$ random variables, with each term is bounded by $2/p(n)n$ and has variance $\leq 4/p(n)n^2$. This means we may apply Bennett's concentration inequality conditionally on the $\omega_i$, with:
\[t = \eta n,\, b:=\frac{2}{p(n)n} \mbox{ and }v\leq \frac{4}{p(n)} .\]
We obtain that for $\eta\leq n$:
\[\Pr{\prodint{\vec{x}}{D\vec{y}}>\eta n}\leq \expp{-\dfrac{\eta^2n^2}{\frac{8}{p(n)} + \frac{4\eta}{3p(n)n}}},\]which is our goal. \end{proof}
\subsection{On exiting compact sets} \label{subsec:proof:le:exitcompact}

In this section we prove Lemma \ref{le:exitcompact}, which bounds the probability that many media variables and/or many diffusions fall outside a large ball. For brevity, we present only the argument for the diffusion system $\theta^{(n)}$ (cf. Definition \ref{def:process} and (\ref{eq:defsystemP})) as the argument for the system $\overline{\theta}^{(n)}$ would be similar. We will mostly drop the $(n)$ superscript from our notation. 
\begin{proof}Recall that
\[E_i(R):= \left\{\sup_{s \in [0,T] }|\theta_i(s)|\geq R/4 \ \mbox{or} \  \ |\omega_i|>R/4\right\} \; (i\in [n]).\]
Our goal is to show that, for fixed $\eta,R,T>0$:
\begin{equation*}\mbox{\bf Goal: }\limsup_{n}\frac{1}{n}\log\Pr{\sum_{i=1}^n\Ind{E_i(R)}>\eta n}\leq -a_T(\eta,R)\end{equation*}
where $a_T(\eta,R)\geq 0$ does {\em not} depend on $n$ and $a_T(\eta,R)\to +\infty$ when $R\to +\infty$. The function $a_T(\eta,R)$ will, however, depend on $\lambda$, $\mu$ and the interaction functions. 

For each $i\in[n]$, $E_i(R)$ is contained in the of the following events: 
\begin{eqnarray}E_{i,1}(R) &=& \{|\omega_i|>R/4\};\\
E_{i,2}(R) &=& \{|\theta_i(0)|>R/8\};\\
\label{eq:defEi3}E_{i,3}(R)&=& \left\{\sup_{s \in [0,T] }|\theta_i(s)-\theta_i(0)|\geq R/8\right\}.\end{eqnarray}
It thus suffices to prove the following claim.

\begin{claim}For each index $c=1,2,3$, and each choice of $\eta,R,T>0$ we have:
\begin{equation*}\limsup_{n\in\N}\frac{1}{n}\log \Pr{\sum_{i=1}^n\Ind{E_{i,c}(R)}>\eta n}\leq -a_{T,c}(\eta,R)\end{equation*}
where $a_{T,c}(\eta,R)\geq 0$ and $a_{T,c}(\eta,R)\to +\infty$ when $R\to +\infty$. 
\end{claim}

Before we prove the claim, we note the a simple general fact. Assume $X_1,\dots,X_m$ are i.i.d. real-valued random variables with common law $P\in\mathcal{M}_1(\R)$. Observe that in particular $P(\{+\infty\})=0$. Then for any $x\geq 0$, $a>0$
\[\Pr{\sum_{i=1}^m\Ind{\{X_i\geq x\}}>am}\leq \binom{m}{\lceil am\rceil}\,P[x,+\infty)^{\lceil am\rceil}.\]
Note that since $P(\{+\infty\})=0$ we have that $P[x,+\infty)\to 0$ when $x\to +\infty$. Combining this with the standard bound:
\[\binom{m}{k}\leq \left(\frac{em}{k}\right)^k\]
gives:
\[\Pr{\sum_{i=1}^m\Ind{\{X_i\geq x\}}>am}\leq \left(\frac{e\,P[x,+\infty)}{a}\right)^{\lceil am\rceil}.\]
In particular, 
\begin{equation}\label{eq:simpleprob}\Pr{\sum_{i=1}^m\Ind{\{X_i\geq x\}}>am}\leq \exp(-b_P(a,x)m)\end{equation}
where $b_P(a,x)$ only depends on $P$, $a$ and $x$ and converges to $+\infty$ as $x\to +\infty$. 

Let us now prove the claim. In the case $c=1$, we may apply (\ref{eq:simpleprob}) directly with $a=\eta$, $m=n$, $x=R/4$ and $X_i=|\omega_i|$. This is because the media variables $\omega_i$ are i.i.d. with a law $\mu$ that does not depend on $n$ and have finite mean (cf. Assumption \ref{ass:medvartail}). Similarly, the claim follows for $c=2$  because the initial conditions $\xi_i=\theta_i(0)$ are also i.i.d. with a law that does not depend on $n$ and have finite first moment.

For the case $c=3$, we go back to the definition of the diffusions as presented in (\ref{eq:defsystemP}). Note that for each $i\in[n]$ and $0\leq t\leq T$, 
\[\theta_i(t) - \theta_i(0) - B_i(t) = \int_{0}^t \left(\psi(\theta_i(s),s) + \sum_{j=1}^n\,P_{i,j}\phi(\theta_i(s),\theta_j(s),\omega_i,\omega_j)\right)\,ds.\]
The functions $\psi,\phi$ are bounded, so:
\[\sup_{0\leq t\leq T}|\theta_i(t) - \theta_i(0)|\leq  \sup_{0\leq t\leq T}|B_i(t)| + C\,T\,(S_i+1).\]
where $C>0$ only depends on $\psi,\phi$ and 
\[S_i:=\sum_{j=1}^nP_{i,j}\,\,(i\in[n]).\]
Therefore, an event $E_{i,3}(R)$ can only hold for a given index $i$ if \[\mbox{either }\sup_{0\leq t\leq T}|B_i(t)|>R/16\mbox{ or }CT\,(S_i+1)>R/16.\] In particular,
 \begin{eqnarray*}\Pr{\sum_{i=1}^n\Ind{E_{i,3}(R)}>\eta n}&\leq & \Pr{\sum_{i=1}^n\Ind{\{\sup_{0\leq t\leq T}|B_i(t)|>R/16\}}>\frac{\eta n}{2}}\\ & & +\Pr{\sum_{i=1}^n\Ind{\{C\,T\,(S_i+1)>R/16\}}>\frac{\eta n}{2}}.\end{eqnarray*}
The first of these terms, 
\[\Pr{\sum_{i=1}^n\Ind{\{\sup_{0\leq t\leq T}|B_i(t)|>R/16\}}>\frac{\eta n}{2}}\] 
has the form in (\ref{eq:simpleprob}) with $X_i=\sup_{0\leq t\leq T}|B_i(t)|$. We may deduce as above that:
\[\Pr{\sum_{i=1}^n\Ind{\{\sup_{0\leq t\leq T}|B_i(t)|>R/16\}}>\frac{\eta n}{2}}\leq \exp(-a'_{T,3}(\eta,R)n)\]
where $a'_{T,3}(\eta,R)\to +\infty$ as $R\to +\infty$. 

To finish, it suffices to show:
\[\Pr{\frac{1}{n}\sum_{i=1}^n\Ind{\{C\,T\,(S_i+1)>R/16\}}>\frac{\eta}{2}}\]
is superexponentially small when $R$ is large enough. To see this, we note that:
\[\frac{1}{n}\sum_{i=1}^n\Ind{\{C\,T\,(S_i+1)>R/16\}}\leq \frac{16CT}{R}\,\frac{\sum_{i=1}^n(S_i+1)}{n}.\]
Letting $\one\in\R^n$ denote the vector with all coordinates equal to $1$, we note that 
\[\sum_{i=1}^nS_i = \prodint{\one}{P\one}.\]
That is, \[\frac{1}{n}\sum_{i=1}^n\Ind{\{C\,T\,(S_i+1)>R/8\}}\leq \frac{16CT}{R}\,\left(\frac{\prodint{\one}{P\one}}{n}+1\right).\]
Now recall from Definition \ref{def:matrices} that $P = \overline{P}+D$ where the entries of $\overline{P}$ are bounded by $\|W\|_{\infty}/n$. So: 
\[\frac{\prodint{\one}{P\one}}{n} = \frac{\prodint{\one}{\overline{P}\one}}{n} + \frac{\prodint{\one}{D\one}}{n}\leq \|W\|_{\infty} + \frac{\|D\|_{\infty\to 1}}{n}.\]
So:
\[\frac{1}{n}\sum_{i=1}^n\Ind{\{C\,T\,(S_i+1)>R/8\}}\leq \frac{16CT}{R}\,\left(\|W\|_{\infty}+1 + \frac{\|D\|_{\infty\to 1}}{n}\right).\]
Therefore, setting:
\[r=r(\eta,R,T):=\frac{R\eta}{32CT} - 1 - \|W\|_{\infty},\]
we obtain
\[\Pr{\frac{1}{n}\sum_{i=1}^n\Ind{\{C\,T\,(S_i+1)>R/8\}}>\frac{\eta}{2}}\leq \Pr{\frac{\|D\|_{\infty\to 1}}{n}>r}.\]
This probability goes to $0$ super-exponentially fast whenever $r\geq 1$, thanks to Lemma \ref{lem:onlymatrix}. This finishes the proof of the claim for $c=3$ and therefore the whole proof.\end{proof}

\appendix

\section{Appendix: an approximation result} \label{subsec:app} \Gnote{subsec:app}

In this subsection we prove the existence of a good approximation as in Lemma \ref{lemma:goodapp}.

\begin{lemma}\label{le:app} Let $\phi:\R^3\to \R$ differentiable. Suppose that there is a constant $M \in \R$ such that $\supnorm{\phi}\leq M$ and $\pnorm{op,\infty}{\nabla\phi}\leq M$. Let $N_3 \in \R^3$ be a normal random variable with mean zero and covariance matrix identity $Id_{3\times 3}$. For any $\varepsilon \in (0,1]$ and for each $\vec{x} \in \R^3$ define $\phi_\varepsilon(\vec{x})=\Ex{\phi(\vec{x}+\varepsilon N_3)}.$   Then
\begin{enumerate}
\item $\phi_{\varepsilon} \in C^\infty (\R^3).$
\item $\supnorm{\phi_\varepsilon}\leq \supnorm{\phi}.$
\item \label{item:goodapp} $\pnorm{\infty}{\phi-\phi_\varepsilon}\leq \varepsilon\pnorm{op,\infty}{\nabla\phi}\Ex{|N_3|}.$
\item $\pnorm{op,\infty}{\nabla\phi_\varepsilon}\leq \pnorm{op,\infty}{\nabla\phi}.$
\end{enumerate}  
\end{lemma}
\begin{proof}
Let $\gamma$ be the density of $N_3$ with respect to the Lebesgue measure. By definition
\begin{eqnarray}
\phi_\varepsilon(\vec{x})&=&\int_{\R^3}\phi(\vec{x}+\varepsilon\vec{y})\gamma(\vec{y})d\vec{y}\nonumber \\
&=&\int_{\R^3}\phi(\vec{z})\gamma\left(\dfrac{\vec{z}-\vec{x}}{\varepsilon^3}\right)d\vec{y}. \nonumber
\end{eqnarray}

Therefore, applying the Convergence Dominated Theorem we can show that
\begin{eqnarray}
\dfrac{\partial \phi_\varepsilon}{\partial x_i}(\vec{x})&=&\int_{\R^3}\phi(\vec{z})\dfrac{\partial\gamma}{\partial x_i}\left(\dfrac{\vec{z}-\vec{x}}{\varepsilon^3}\right)d\vec{y} \nonumber
\end{eqnarray}
and the same is true for all higher derivatives.  Therefore, $\gamma \in C^{\infty}(\R^3)$ implies $\phi_\varepsilon \in C^{\infty}(\R^3).$

Again by the Convergence Dominated Theorem, using that $\phi$ has one derivative
\begin{eqnarray}
\dfrac{\partial \phi_\varepsilon}{\partial x_i}(\vec{x})&=&\int_{\R^3}\dfrac{\partial\phi}{\partial x_i}(\vec{x}+\varepsilon\vec{y})\gamma(\vec{y})d\vec{y}.\nonumber
\end{eqnarray}
This implies that $\pnorm{op,\infty}{\nabla \phi_\varepsilon}\leq \pnorm{op,\infty}{\nabla \phi}$. For the third claim we write
\begin{eqnarray}
\phi_\varepsilon(x)-\phi(x)=\Ex{\phi(x+\varepsilon N)-\phi(x)} \nonumber 
\end{eqnarray}
to see that the Mean Value Theorem implies $\supnorm{\phi_\varepsilon-\phi}\leq \varepsilon\pnorm{op,\infty}{\nabla \phi}\Ex{|N|}.$
\end{proof}

The next Lemma implies Lemma \ref{le:app} in the main text. We will need a bump function $\xi$, that is, a $C^{\infty}$ function such that
\begin{itemize}
\item $\pnorm{\infty}{\xi}\leq 1$.
\item $\pnorm{\infty}{\xi'}\leq C_1$ (a constant that does not depend in any parameter).
\item $\xi\equiv 1$ in $[-1,1]$.
\item $\xi \equiv 0$ in $[-2,2]^c$.
\end{itemize}

\begin{lemma}\label{le:goodapp} \Gnote{le:goodapp}Consider $\phi$ and $\phi_\varepsilon$ as in Lemma \ref{le:app} . Define also for all $R\geq 1$ and $\vec{x} \in \R^3$ $$\phi_{\varepsilon,R}(\vec{x})=\phi_\varepsilon(\vec{x})\xi\left(\dfrac{\pnorm{2}{\vec{x}}^2}{R^2}\right).$$ Then
\begin{enumerate}
\item $\phi_{\varepsilon,R} \in C^\infty(\R^3).$
\item \mbox{supp}.$\phi_{\varepsilon,R} \subset B_{2R}(\vec{0}).$
\item $\supnorm{\phi_{\varepsilon,R}}\leq \supnorm{\phi}.$
\item $\pnorm{op,\infty}{\nabla\phi_{\varepsilon,R}}\leq \pnorm{op,\infty}{\nabla\phi}+\supnorm{\xi'}\supnorm{\phi}.$
\item $\pnorm{L^{\infty}(B_{R}(\vec{0}))}{\phi_{\varepsilon,R}-\phi}\leq \varepsilon \pnorm{op,\infty}{\nabla\phi}\Ex{|N_3|}.$
\end{enumerate}
In this way we choose $$M=\max\{\supnorm{\phi},\pnorm{op,\infty}{\nabla\phi}+\supnorm{\xi'}\supnorm{\phi},\pnorm{op,\infty}{\nabla\phi}\Ex{|N_3|}\}$$ and write $\phi^{\eps,R}:=\phi_{\eps/M,R}$ to state \ref{lemma:goodapp} .
\end{lemma}

\begin{proof} Items $1-3$ are immediate from the definition of $\xi$ and $\phi_\varepsilon.$ To check item 4 we apply the product rule to obtain
\begin{eqnarray}
\dfrac{\partial\phi_{\varepsilon,R}}{\partial x_i}(x)=\dfrac{\partial\phi_\varepsilon}{\partial x_i}(x)\xi\left(\dfrac{x}{R}\right)+\phi_\varepsilon(x)\xi'\left(\dfrac{\pnorm{2}{\vec{x}}^2}{R^2}\right)\dfrac{2x_i}{R^2}. \nonumber
\end{eqnarray}
For the first term on the right hand side remember that $\xi \leq 1$ and $\pnorm{op,\infty}{\nabla\phi_\varepsilon}\leq\pnorm{op,\infty}{\nabla\phi} $. The second term vanishes when $\pnorm{2}{\vec{x}}\geq R$ since supp.$\xi\subset [-1,1]$. In the case $\pnorm{2}{\vec{x}}< R$ we have that $$\modulo{\dfrac{2x_i}{R^2}}\leq 1.$$ To finish item 4 remember that $\pnorm{2}{\vec{x}}\leq \pnorm{1}{\vec{x}}$ in such way that we just need to sum the last bounds. 

To check item 5 we just need to note that $\phi_{\varepsilon,R}=\phi_\varepsilon$ in $B_R(\vec{0})$ and use item \ref{item:goodapp} of Lemma \ref{le:app}.\end{proof}

\section{Appendix: extension of the ``dense"~LDP} \label{sec:extendingg} \Gnote{sec:extendingg}

In this Appendix we check that the same large deviations result and McKean-Vlasov limit obtained by dai Pra and den Hollander \cite{Pra1996} hold in our slightly more general setting. More specifically, we wish to sketch a proof of the following result. 

\begin{theorem}Consider the sequence of empirical measures $\{\overline{L}_n\}_{n\in\N}$ under Assumptions \ref{ass:medvartail}, \ref{ass:edge} and \ref{ass:fg}. Then $\{\overline{L}_n\}_{n\in\N}$ satisfies a Large Deviations Principle with the rate function $I$ in Definition \ref{def:rate}, which has a unique McKean-Vlasov diffusion as minimizer.\end{theorem}

We review the points we discussed in Remark \ref{rem:LDPdifferences}. The trajectories in $\overline{\theta}^{(n)}$ evolve according to the Hamiltonian 
\[\overline{H}_n(x^{(n)},\omega^{(n)}):= \frac{1}{2n}\sum_{i,j=1}^n\,\overline{f}(x^{(n)}_i-x^{(n)}_j,\omega_i^{(n)},\omega_j^{(n)}) + \sum_{i=1}^n\,g(x^{(n)}_i,\omega_i^{(n)}).\]
This is the same kind of Hamiltonian in \cite{Pra1996}, except that $f$ is replaced by $\overline{f}$. 

Our assumptions on the measures $\mu$ and $\lambda$ are the same as in \cite{Pra1996}. The assumptions on $\overline{f}$ and $g$ are nearly the same as in \cite{Pra1996}, but we only assume $\overline{f}',\overline{f}'',g',g''$ are bounded Lipschitz, whereas \cite{Pra1996} also requires that $\overline{f},g$ be bounded. 
 
We now explain how to adapt the proofs of Lemma 1, Theorem 1 and Theorem 2 in \cite{Pra1996} to our slightly weaker assumption. One important point is that $\overline{f}=\overline{f}(x,\omega,\pi)$ and $g=g(x,\omega)$ are $L$-Lipschitz in the first variable, with a constant $L>0$ that does not depend on $\omega$ or $\pi$. In particular, Lemma 1 in their paper, which describes the law of $\overline{L}_n$ as an exponential tilt, works exactly the same way as in their paper, via Girsanov's Theorem and It\^{o}'s Formula. 
\begin{eqnarray}
P_N(\cdot)=\int d(W^{\otimes N}\otimes \mu^{\otimes N})\expp{NF(L_N)}\Ind{\{L_N \in \cdot\}}
\end{eqnarray}

Theorem 1 uses the exponential tilting argument to derive a LDP for $\overline{L}_n$. This requires a slight amount of care, as the tilting functional $F$ is unbounded in our setting. However, the fact that $f,g$ are Lipschitz implies:
\[|F(L_N)|\leq K\,\left(1+\int\,|x_T-x_0|\,L_N(dx_{[0,T]}d\omega)\right)\]
for some constant $K>0$. Thus the exponential integrability conditions in Varadhan's Lemma (cf. \cite[Theorem 4.3.1]{dembo2009large}) apply and allow us to conclude the proof. 

For Theorem 2, the main body of the proof follows in the same way from It\^{o}'s Formula. The only change is in the argument for uniqueness in Appendix A. More specifically, what we need to do (in their notation) is show that the density of $Q_*$ at time $t$ conditionally on $\omega$ satisfies a bound:
\[q_t^{\omega}(z)\leq B_T\,t^{-\alpha}\]
with $0\leq \alpha<1/2$ and $B$ independent of $\omega$ (but may depend on $T$).

To obtain this, the \cite{Pra1996} uses the boundedness of $f$ and $g$ when they claim that the drift $\beta_t^{\omega,\Pi_tQ_*}$ is the bounded derivative of a bounded function. In our case the drift is a bounded derivative of a Lipschitz function. Therefore, for any event $E\subset C([0,T],\R)\times \R$,
\[Q_*(E) = \int_{A}\,Z_T\,W_\lambda\otimes \mu(dx_{[0,T]}d\omega)\]
where $|\log Z_T|\leq K\,(1+|x(T)-x(0)|)$ and $W_\lambda$ is the law of Brownian motion started from measure $\lambda$. Now if $E$ takes the form:
\[E:= \{(x_{[0,T]},\omega)\,:\, x(T)\in A,\omega\in B\},\]then:
\[Q_*(E) \leq \left(\int_{\R^2} e^{K\,(1+|x|)}\,\Ind{A}(x+y)\rho_{t}(x)\phi(y)\,dx\,dy\right)\times \mu(B),\]
where $\rho_{t}$ is the density of a $N(0,t)$ random variable and $\phi$ is the density of the initial measure $\lambda$. Using the notation of their paper, we obtain:
\[q_t^{\omega}(z)\leq \int_\R\,e^{K\,(1+|z-y|)}\,\phi(y)\,\rho_t(z-y)\,dy.\]
We may apply H\"{o}lder's inequality as in their proof to obtain:
\[q_t^{\omega}(z)\leq \|\phi\|_{L^p}\,\left(\int_\R\,e^{K\,q(1+|z-y|)}\,\rho_t(z-y)^q\,dy\right)^{\frac{1}{q}}\leq B\,t^{(1/2 -q/2)}.\]




%
%

\bibliographystyle{plain} 

\bibliography{IntDiff-not-too-Sparse}

\end{document}